\documentclass[12pt,a4paper]{amsart}

\renewcommand\rightmark{Geometry of solutions
  to the metrizability equation}

\usepackage{amssymb}

\usepackage{amscd}
\usepackage{hyperref}

\def\Cal{\mathcal}

\let\phi\varphi

\newcommand{\into}{\hookrightarrow}

\newcommand{\al}{\alpha}

\renewcommand{\th}{\theta}
\newcommand{\si}{\sigma}
\newcommand{\ze}{\zeta}
\newcommand{\Ga}{\Gamma}

\newcommand{\bdet}{\textbf{det}}

\newcommand{\cD}{\mathcal{D}}
\newcommand{\cE}{\mathcal{E}}
\newcommand{\cT}{\mathcal{T}}
\newcommand{\rank}{\operatorname{rank}}

\newcommand{\rpl}                         
{\mbox{$
\begin{picture}(12.7,8)(-.5,-1)
\put(0,0.2){$+$}
\put(4.4,3.1){\oval(8,8)[r]}
\end{picture}$}}

\usepackage{amssymb, amsfonts, amsmath, amsthm, bbm} 
\usepackage{colonequals} 
\usepackage{comment}
\usepackage{mathtools}
\usepackage[australian]{babel}
\usepackage{mathrsfs}
\usepackage{tensor}
\usepackage{tikz-cd}
\usepackage{rotating}
\usetikzlibrary{matrix,arrows}
\usepackage[margin=1in]{geometry}
\newcommand*{\isoarrow}[1]{\arrow[#1,"\rotatebox{90}{\(\sim\)}"
]}
 \usepackage[T1]{fontenc}

\newcommand{\thorn}{\text{\slshape\itshape\th}}

\newtheorem{theorem}{Theorem}[section]
\newtheorem{lemma}[theorem]{Lemma}
\newtheorem{proposition}[theorem]{Proposition}
\newtheorem{corollary}[theorem]{Corollary}
\newtheorem{thm-and-def}[theorem]{Theorem and Definition}

\theoremstyle{definition}

\newtheorem{remark}[theorem]{Remark}
\newtheorem{notation-and-rem}[theorem]{Notation and Remark}

\newcommand{\bp}{\boldsymbol{p}}
\newcommand{\bc}{\boldsymbol{c}}

\def\sideremark#1{\ifvmode\leavevmode\fi\vadjust{\vbox to0pt{\vss
 \hbox to 0pt{\hskip\hsize\hskip1em
 \vbox{\hsize3cm\tiny\raggedright\pretolerance10000
  \noindent #1\hfill}\hss}\vbox to8pt{\vfil}\vss}}}%

\begin{document}
\title{Metrics in projective differential geometry: the geometry of solutions
  to the metrizability equation}

\author{Keegan Flood and A.\ Rod Gover}

\address{K.F. \& A.R.G.:Department of Mathematics\\
  The University of Auckland\\
  Private Bag 92019\\
  Auckland 1142\\
  New Zealand} 
\email{keegan.flood@auckland.ac.nz}
\email{r.gover@auckland.ac.nz}

\begin{abstract} 
Pseudo-Riemannian metrics with Levi-Civita connection in the
projective class of a given torsion free affine connection can be
obtained from (and are equivalent to) the maximal rank solutions of a
certain overdetermined projectively invariant differential equation
often called the metrizability equation. Dropping this rank assumption
we study the solutions to this equation given less restrictive generic
conditions on its prolonged system. In this setting we find that the
solution stratifies the manifold according to the strict signature
(pointwise) of the solution and does this in way that locally
generalizes the stratification of a model, where the model is, in each
case, a corresponding Lie group orbit decomposition of the
sphere. Thus the solutions give curved generalizations of such
embedded orbit structures. We describe the smooth nature of the strata
and determine the geometries of each of the different strata types;
this includes a metric on the open strata that becomes singular at the
strata boundary, with the latter a type of projective infinity for the
given metric. The approach reveals and exploits interesting highly non-linear relationships
between different linear geometric partial differential equations.
Apart from their direct significance, the results show
that, for the metrizability equation, strong results arising for so-called normal
BGG solutions, and the corresponding projective holonomy reduction, extend to a far wider
class of solutions. The work also provides new results for the
projective compactification of scalar-flat metrics.
\end{abstract}


\thanks{Both authors gratefully acknowledge support from the Royal
  Society of New Zealand via Marsden Grants 13-UOA-018 and 16-UOA-051.}

\maketitle

\section{Introduction}\label{intro}

On geometric manifolds the natural overdetermined partial differential
equations govern a variety of key phenomena including symmetry (such
as the Killing equation on infinitesimal isometries), so-called hidden
symmetries, and also many equations directly governing geometric
structure \cite{BMMR,CG2,EM,GM,Jez,Pap}. It turns out that a solution
of such an equation can often stratify the manifold in an important
way. A problem of classical interest is to determine the possible zero
locus of solutions of the Killing equation, the conformal Killing
equation, and related symmetry equations
\cite{BL,Derd,Leit2,Lis}.
In cases there is some
relation to the determination of nodal sets (as for eigenfunctions of
Laplacian cf. \cite{Z3,Z1,Z2}) but in general there are also important
differences because of the greater number of equations controlling the
solution. Indeed, for a given overdetermined PDE and solution thereof,
there is potentially very rich information available; for example
  the various strata on the same given manifold can
encode different geometries that are strikingly different (see
e.g.\ \cite{GPW}). Evidently, in such cases the solution smoothly
relates these different geometries and so can be used as a tool for
studying one in terms of the other in the spirit of the geometric
holography program (cf.\ \cite{FG,GLW,GZ,Juh,MM}). 
Thus given a
particular overdetermined partial differential equation on a given
manifold the first important problems for solutions (or perhaps some
distinguished class of solutions) are: (i) To determine the nature of
the strata, e.g. are they smoothly embedded submanifolds of some
dimension or rather more complicated variety type structures? (ii) To
determine in detail any geometries induced on the different
strata. (iii) To understand how the geometries on neighboring strata
are related.

Toward capturing the nature of the strata, the geometries they
determine, as well as the relation between these,  rather general
results are available in \cite{CGH1} and \cite{CGH2}. For a vast class of
overdetermined linear partial differential equations and solutions
thereof, the so-called {\em normal solutions} of {\em first BGG equations}, these sources show
that the stratifications must be locally diffeomorphic to
stratifications arising in simpler model cases and moreover the
different strata carry Cartan geometries that are, in a precise way,
curved analogues of the Klein (i.e. homogeneous) geometries on the
corresponding strata of the model.  However at this stage it seems the approach in these sources does
not extend beyond these special normal solutions (these are
solutions that correspond to suitable Cartan holonomy reductions). The
question then arises as to whether similar results might be available
for more general solutions.

On an $n$-manifold $M$ with an affine connection $\nabla$ an interesting question is
whether there is a metric on $M$ with the same geodesics, up to
reparametrization, as $\nabla$. Here and throughout $n\geq 2$.
It is a result of Mikes and Sinjukov
 \cite{Mik,Sinj} that this non-linear problem
can be recast in terms of an equivalent linear PDE problem: there is such a
metric if and only if there is a rank-$n$ symmetric contravariant
2-tensor $\zeta^{bc}$ that satisfies the equation

\begin{equation}\label{meteq}
\operatorname{trace-free} (\nabla_a \zeta^{bc}) = 0,
\end{equation}

\noindent where we employ an obvious (abstract) index notation. In the
case that there exists such a full rank solution of (\ref{meteq}) then
the corresponding (inverse) metric is given by
$g^{bc}=\operatorname{sgn}(\tau) \tau \zeta^{bc}$, where $\tau
\colonequals \bdet(\ze)$ is a suitable determinant of $\zeta$ defined
in expression (7). This equation and surrounding questions have been
the subject of intense recent interest and considerable progress
\cite{BM, BDE, DE, EM, FM, GM, KM, Mat2, MR, Mettler}, and there is
growing interest in the related c-projective analogue see
e.g. \cite{CEMN, Mettler2}.

In the current article we study the solutions of this {\em
  metrizability equation} (\ref{meteq}). Given its interpretation the
first important issue for any solution $\zeta$ is the nature of its
{\em degeneracy locus} $\cD(\zeta)$, that is the set of points where
the rank of $\zeta$ is less than $n$. On such a set there is in
general no metric, but, as we shall see, (given mild restrictions)
there is interesting geometry and one of our aims is to determine this
and also an understanding of how it arises from the  ambient metric
which is available on the open set where $\zeta$ has maximal
rank. This ambient metric is singular along $\cD(\zeta)$.  The problem
is of direct interest because of strong links with the program of
projectively compactifying complete non-compact Riemannian and
pseudo-Riemannian metrics, as developed and studied in
\cite{CG1,CG2,CG3}. Indeed, although ostensibly we study a different
problem, the current article provides a new perspective on the
projective compactification of metrics with scalar curvature pointwise
bounded away from zero and strong new results for the projective
compactification of scalar-flat metrics (see Corollary \ref{4.16}). In
addition to these motivations, the equation (\ref{meteq}) is also an
important ``test case'' for the general problem mentioned in the first
paragraph.

Throughout affine connections will be assumed torsion free. Two such
connections $\nabla$ and $\nabla'$ are said to be {\em projectively
  related} if they share the same geodesics as unparametrized
curves. An equivalence class $\bp:=[\nabla]$ of such connections is
termed a {\em projective structure} and a manifold $M$ equipped with
such an equivalence class, written $(M,\bp)$, is a {\em projective
  manifold}. The equation (\ref{meteq}) is a {\em projectively
  invariant} meaning that, when interpreted correctly, it descends to
a well defined equation on projective manifolds $(M,\bp)$, even though 
there is in general no distinguished affine connection in $\bp$.

Overdetermined equations are typically best studied by some form of
differential prolongation where new variables are introduced to
produce a first order closed system (see e.g. \cite{BCEG}).  Because equation
(\ref{meteq}) is projectively invariant this prolonged system is
handled naturally by the projective tractor calculus as presented in \cite{BEG}. In
fact (\ref{meteq}) falls into the class of {\em first BGG} equations
\cite{CSS,CSouc,CGH1,CGH2}. Associated with any first BGG
equation there is a canonical invariant differential operator called a
(first) BGG splitting operator which, informally speaking, maps the
domain section to its prolonged variable system.  In particular in
this case there is a projectively invariant second order operator
$\zeta\mapsto L(\zeta)$ and $L(\zeta)$ takes values in the second
symmetric power of the standard projective tractor bundle
$S^2\cT$. The solution $\zeta$ is {\em normal} if $L(\zeta)$ is
parallel for the tractor connection, but here we do {\em not} restrict to
normal solutions.  These objects are introduced in Sections \ref{background} and \ref{BGGsec} below, but the important thing at this stage is that they are canonically associated to the projective manifold and on an $n$-manifold the standard tractor bundle
$\cT$ has rank $n+1$. For sections of $S^2\cT$ there is a canonical
(projectively invariant) determinant available and so it is natural to
consider the composition of this with the $L(\zeta)$
\begin{equation}\label{det}
\zeta\mapsto L(\zeta)\mapsto \det L(\zeta). 
\end{equation}
Now a key point. If $\zeta$ is a maximal rank solution of
(\ref{meteq}) then $\det (L(\zeta))$ is, up to a non-zero constant, a
multiple of the scalar curvature of the corresponding metric $g$
with inverse $g^{-1}=\operatorname{sgn}(\tau) \tau \zeta$
\cite{CG3}. However the determinant (\ref{det}) is well defined even
where $\zeta$ is not of maximal rank. Thus it is natural to consider
solutions $\zeta$ of equation (\ref{meteq}) satisfying the 
condition that $\det L(\zeta)$ is nowhere zero, i.e.  with $L(\zeta)$
of maximal rank, but with no {\em a priori} restriction on the rank of
$\zeta$. Note that this is a generic condition. Furthermore it is a
generalization of constant scalar curvature, but where $\zeta$ is
allowed to have a non-trivial degeneracy locus. With $\tau=\bdet(\ze)$, as above,
we obtain the following result.

\begin{theorem}\label{one} Let $(M,\bp)$ be an $n-$dimensional projective manifold equipped with a solution $\ze^{ab}$ of the metrizablility equation such that $\mathcal{R}(L(\zeta))=n+1$.  If $L(\ze)$ is definite then the degeneracy locus $\mathcal{D}(\ze)$ is empty and $(M, \bp,\zeta)$ is
  a Riemannian manifold with inverse metric $g^{-1}=\operatorname{sgn}(\tau)\tau
  \ze$. If $L(\ze)$ has signature $(p+1,q+1)$, with $p,q\geq 0$, then $\mathcal{D}(\ze)$ is either empty or it is a smoothly embedded separating hypersurface such that the following hold: 
  \\
  (i) $M$ is stratified by the strict signature of $\ze$ as a (density
  weighted) bilinear form on $T^*M$ with the partition of $M$ given by
$$
M=  \coprod\limits_{i \in \{+,0,-\}} M_i
$$
 where $\zeta$ has signature $(p+1,q)$,
$(p,q+1)$,and $(p,q,1)$ on $M_{+}$, $M_{-}$, and $M_0$, respectively. \\
(ii) $M_0$ has a conformal structure of signature $(p,q)$. \\
(iii) On $M_\pm$, $\zeta$ induces a pseudo-Riemannian metric $g_{\pm}$, of the same signature as $\zeta$, with inverse $g^{-1}_{\pm} = \operatorname{sgn}(\tau) \tau \zeta|_{M_{\pm}}$. \\
(iv) If $M$ is closed, then the components $(M \backslash M_{\mp},\bp)$ are 
order 2 projective compactifications of $(M_{\pm}, g)$, with boundary $M_0$.\\
 \end{theorem}
\noindent Theorem \ref{one} is a summary of the results obtained in Lemma
\ref{4.5} and Theorem \ref{4.6}. With reference to the first statement in the
Theorem, note that if $L(\ze)$ is negative definite then the
corresponding metric $g$ is also negative definite.  We denote the
signature of a real symmetric bilinear form by $(p,q,r)$, where $p,q$
and $r$ are the number, counting multiplicity, of positive, negative,
and zero eigenvalues, respectively, of any matrix representing the
form once a basis has been chosen. When $r=0$ we omit it. We define a
hypersurface to be a smoothly embedded submanifold of codimension 1.

Next to make contact with scalar-flat metrics we must consider
solutions $\zeta$ with $\det L(\zeta)=0 $. On the other hand $\rank
(\zeta)\leq \rank(L(\zeta))$ so the case of interest is
$\rank(L(\zeta))=n$. Note that this is a generic case among solutions
with $\det L(\zeta)=0 $. In this setting the geometries involved
differ to those above, and there can be a finer stratifcation:

\begin{theorem}\label{two}
 Let $(M,\bp)$ be a projective 
manifold equipped with a solution $\ze$ of the metrizablility 
equation such that $L(\ze)$ has signature $(p,q,1)$. If $\mathcal{D}(\ze)=\varnothing$, then $\zeta$ induces a scalar-flat pseudo-Riemannian metric of signature $(p,q)$ on $M$. Otherwise, if $\varnothing \subsetneq\mathcal{D}(\ze)\subsetneq M$, then the following hold:

\noindent $(i)$ $\mathcal{D}(\ze)$ is a smoothly embedded hypersurface. If $M$ and $\mathcal{D}(\ze)$ are orientable then $\mathcal{D}(\ze)$ is separating and $M$ is stratified according to the strict sign of a canonical projective density, $\si$, that is locally a square 
root of ${\normalfont \bdet}(\ze)$ or $-{\normalfont \bdet}(\ze)$. The partition of $M$ is given by  
$$
M=  \coprod\limits_{i \in \{+,0,-\}} M_i
$$ with $\si >0$ on $M_+$ and $\si <0$ on $M_-$, and $\si = 0$ on $M_0=\mathcal{D}(\ze)$. 

\noindent $(ii)$ $M_0$ is totally geodesic and inherits a projective structure $\hat{\bp}$. 

\noindent $(iii)$ On $M_\pm$, $\zeta$ induces a scalar-flat pseudo-Riemannian metric $g_{\pm}$, of the same signature as $\zeta$, with inverse $g^{-1}_{\pm} = \operatorname{sgn}(\tau) \tau \zeta|_{M_{\pm}}$, where $\tau\colonequals \text{\normalfont \bdet}(\ze^{ab})$. If $M$ is closed, then the components $(M \backslash M_{\mp},\bp)$ are order 1 projective compactifications of $(M_{\pm}, g)$, with boundary $M_0$.

\noindent $(iv)$ $(M_0,\hat{\bp})$ inherits a solution $\hat{\zeta}=\ze|_{\Sigma}$ of the metrizability equation and $\Sigma \colonequals M_0$ decomposes into
$$
\Sigma=  \coprod\limits_{i \in \{+,0,-\}} \Sigma_i
$$
according to the strict signature of $\hat{\zeta}$, where $\Sigma_+$ and $\Sigma_{-}$ are the components with $\hat{\ze}$ of signature $(p,q-1)$ and $(p-1,q)$, respectively.
Further $\Sigma_0$ inherits a conformal structure $(\Sigma_{0}, \bc)$ of signature $(p-1,q-1)$.\\
 \end{theorem} 
\noindent The results in Theorem \ref{two} hold locally regardless of
orientability of $M$ and $\mathcal{D}(\ze)$.  The components $M_+$,
$M_0$, and $ M_{-}$ in the above theorems are not necessarily each
connected. Theorem 1.2 is a summary of the results obtained from
Theorem \ref{4.11}, Proposition \ref{4.12}, and Theorem \ref{4.14}. An
interesting feature of the development of these results is that it
involves a detailed treatment of highly non-linear relationships
between different linear geometric partial differential equations.

In \cite{CG3} it is shown that if the interior of a manifold with
boundary is equipped with a pseudo-Riemannian metric satisfying a
non-vanishing scalar curvature condition and whose Levi-Civita
connection does not extend to the boundary, while its projective
structure does, then the metric is projectively compact of order
2. From Theorem \ref{two} follows an analogue of that result for
metrics of zero scalar curvature.

\begin{corollary}\label{mcor}
Let $\overline{M}$ be an orientable, connected manifold with boundary $\partial M$ and interior $M$, equipped with a scalar-flat pseudo-Riemannian metric $g$ on $M$, such that its Levi-Civita connection $\nabla^{g}$ does not extend to any neighborhood of a boundary point, but the projective structure $\bp \colonequals [\nabla^g]$ does extend to the boundary. Let $\tau \colonequals \operatorname{vol}(g)^{-\frac{2}{n+2}}$. Then $\ze^{ab}\colonequals \tau^{-1}g^{ab}$ extends to the boundary. If $L(\ze^{ab})$ has rank $n$ on $\overline{M}$, then $(M,g)$ is projectively compact of order $1$.
\end{corollary}
\noindent The condition that $L(\ze^{ab})$ have rank $n$ on $\overline{M}$
implies that the scalar curvature is identically zero on $M$.

Insight
and further motivation for the work here is provided by the models for
these structures. Just as the usual round sphere is a compact
homogeneous model for Riemannian geometry there are corresponding
compact models for the structures captured in Theorems \ref{one} and
\ref{two}, as follows.

The standard homogeneous model for projective geometry is the
$n$-sphere arising as the ray projectivization
$S^{n}=\mathbb{P}_+(\mathbb{R}^{n+1})$ of $\mathbb{R}^{n+1}$
(i.e.\ the double cover of $\mathbb{RP}_{n+1}$). The unparametrized
geodesics are the embedded great circles. On this the group
$G=SL(\mathbb{R}^{n+1})$ acts transitively. Now suppose we fix on
$\mathbb{R}^{n+1}$ a non-degenerate symmetric bilinear form $h$ of
signature $(p+1,q+1)$. In $G$ consider the subgroup $H:=SO(h)\cong
SO(p+1,q+1)$ fixing $h$ (so $p+q=n-1$). This acts on the projective
sphere $S^n$ but now with orbits parametrized by the strict sign of
$h(X,X)$ where $X$ denotes the homogeneous coordinates of a given
point on $S^n$. The projective sphere $S^n$ equipped with this action
of $H$ and accompanying orbit decomposition is the model for the
structure discussed in Theorem \ref{one}. This follows easily from the
tractor approach that we use with the interpretation of the tractor
bundles over the homogeneous space $G/P$.  (So the Theorem also
reveals, for this model, the general features of the orbits and the
geometries thereon.) In fact, $h^{-1}=L(\zeta)$ where $\zeta$ is the
corresponding solution of (\ref{meteq}) and, in the language of
\cite{CSlov}, this is a holonomy reduction of a flat Cartan geometry
(namely $G\rightarrow S^n$). Turning this around we see that the
Theorem \ref{one} shows that solutions $\zeta$ of equation
(\ref{meteq}), satisfying that $\det (L(\zeta)) $ is nowhere zero,
provide well behaved curved generalizations of this model even though
$\zeta$ is not required to be normal (i.e. $L(\zeta)$ is not required
to be parallel).


Next consider again $S^{n}=\mathbb{P}_+(\mathbb{R}^{n+1})$ and acting
on this the group $G$ as above. Consider now a rank $n$ symmetric
bilinear form $k$ on $(\mathbb{R}^{n+1})^*$, of signature $(p,q,1)$,
and a covector $0\neq u\in(\mathbb{R}^{n+1})^* $ satisfying
$k(u,\cdot)=0$. The subgroup $H<G$ simultaneously fixing $k$ and $u$
is a copy of the pseudo-Euclidean group $SO(p,q)\rtimes \mathbb{R}^n$,
and $S^n$ with this action is the model for the structure treated in
Theorem \ref{two}. In this case $L(\zeta)=k$ where $\zeta$ is a
corresponding solution of (\ref{meteq}), again these claims follow
easily from the general theory in \cite{CGH2}, namely that each
component of the manifold decomposition corresponds to an orbit on the
model, together with our results in Section \ref{dsec}. Thus Theorem
\ref{two} shows that solutions to (\ref{meteq}) with $\rank
(L(\zeta))=n$, at all points, are curved generalizations of this
model. These are well behaved in the spirit of the results in
\cite{CGH1,CGH2} but without the assumption of solution
normality. Furthermore, the corollary shows that we obtain a
projective compactification that generalizes the model case. In the
model case we identify both the lower and upper hemispheres of
$\mathbb{S}^n$, via central projection, with indefinite
pseudo-Euclidean $n$-space, $\mathbb{E}^{(p,q)}$. Then, via this
construction, the boundary of projectively compact pseudo-Euclidean
space is identified with the closed equatorial $H$ orbit,
$\mathbb{S}^{n-1}$, which is itself a lower dimensional copy of the
model discussed previously, and hence decomposes into $SO(p,q)$
orbits.

This result, that the Theorems show the structures we consider
generalize in a very precise way these orbit decompositions,
demonstrates that the structures we consider (i.e. projective
manifolds equipped with solutions of \eqref{meteq} satisfying the
given constant rank conditions on their prolonged systems $L(\zeta)$)
are sound and interesting. The existence of curved examples follows at
once from the examples of projectively compactified metrics discussed in \cite{CG1,CG2}. The
assumption in Theorems \ref{one} and \ref{two} that $L(\zeta)$ has
constant rank ($n+1$ and $n$ respectively) is, in the language of
\cite{CGH2,CGH1} a {\em constant $G$-type} assumption. (On connected
manifolds this is clearly automatic for normal solutions.)  On a
Riemannian manifold the scalar curvature can be locally almost any
function, as is clear from the results Kazdan and Warner on prescribed
scalar curvature \cite{KW1,KW2}. This shows that there are solutions
of equation (\ref{meteq}) where the rank of $L(\zeta)$ moves between
$(n+1)$ and $n$ in a very complicated manner. So it would seem that
the fixed-rank $G$-type assumptions are necessary to get a reasonable
theory.


The structure of the article is as follows. In the Section \ref{background} we
briefly review projective tractor calculus and projective
compactification. These provide the framework and computational tools
we  utilize. In Section \ref{BGGsec} we describe
BGG machinery and develop two examples that are relevant to our later
results. Finally, in Section \ref{results}, we state and prove the main results.

\section{Projective Tractor Calculus and Projective Compactification}\label{background}
\setcounter{section}{2}

Let $M$ be a manifold of dimension $n \geq 2$, equipped with $\bp$, a
projective class of torsion free affine connections. Then the pair
$(M, \bp)$ is called a {\em projective manifold}. Connections in the
projective class $\bp$ have the same geodesics up to reparametrization
(i.e. as unparametrized curves).  Two such connections $\nabla,
\overline{\nabla} \in \bp$ are explicitly related by the formula
\begin{align*}
\overline{\nabla}_{a} Y ^{b} = \nabla _{a} Y ^{b} + \Upsilon _{a} Y ^{b} +  \Upsilon _{c} Y ^{c} \delta ^{b}_{a}, 
\end{align*}
and its dual
\begin{align*}
\overline{\nabla}_{a} u _{b} = \nabla _{a} u _{b} - \Upsilon _{a} u
_{b} - \Upsilon _{b} u _{a}\end{align*} for $Y \in \Gamma (TM)$, $u
\in \Gamma (T^{*} M)$, and for some one-form $\Upsilon \in \Gamma
(T^{*} M)$. The indices in the above formulae are abstract indices.

We will use Penrose abstract index notation when convenient. So for
example $\mathcal{E}^a$ and $\mathcal{E}_a$ are alternative notations
for $TM$ and $T^*M$ respectively. Contraction is indicated by repeated
indices in the usual way. We symmetrize over abstract indices
contained in parentheses and skew over indices contained in square brackets,
e.g. $T_{(ab)} = \frac{1}{2}(T_{ab}+T_{ba})$ and $S_{[ab]} =
\frac{1}{2}(S_{ab}-S_{ba})$.

In our treatment an important role is
played by the links between metrics and projective structure.  For
details on metrics and Einstein metrics in projective geometry see
e.g. \cite{CGM,EM,GLW2,GH,GM}.

\subsection{Tractor bundles and tractor connections}\label{tbundles} 
The basic invariant calculus on projective manifolds is the so-called projective tractor calculus \cite{BEG,CGM} and we briefly  recall this here. 
Let $\mathcal{E}(1)$ be the $(2n + 2)^{th}$ root of the naturally
oriented lined bundle $(\Lambda ^{n} TM)^{2}$, and by $\mathcal{E} (w)
\coloneqq (\mathcal{E} (1))^{w}$ we will denote the $w^{th}$ power of
$\mathcal{E} (1)$, for $w \in \mathbb{R} $.  Then we write
$\mathcal{B} (w) \coloneqq \mathcal{B} \otimes \mathcal{E} (w)$ for
any bundle $\mathcal{B}$.
Note that any affine connection acts on $(\Lambda ^{n} TM)^{2}$ and hence on its roots $\mathcal{E} (w)$.

The {\em cotractor bundle}
${\mathcal{E}}_{A}$ is defined by
\begin{align*}
{\mathcal{E}}_{A} \colonequals J^{1}( \mathcal{E}  (1)),
\end{align*}
where $J^1( \mathcal{E} (1))$ denotes the bundle of $1$-jets of
sections of $ \mathcal{E} (1)$. The short exact sequence, usually called the jet exact sequence at 1-jets, 
\begin{align}\label{ceuler}
0 \rightarrow \mathcal{E}_b(1) \rightarrow J^1(\mathcal{E}(1)) \rightarrow \mathcal{E}(1) \rightarrow 0
\end{align}
describes the filtration structure on the cotractor bundle, in the
sense that there is a subbundle $\mathcal{B} \subseteq \mathcal{E}_A$,
such that $\mathcal{B} \cong \mathcal{E}_a(1)$ and $\mathcal{E}_A /
\mathcal{B} \cong \mathcal{E}(1)$. A connection on $\mathcal{E}(1)$ is
the same as a splitting of the sequence \eqref{ceuler}, so we will
sometimes refer to a choice of connection in the projective class as a
{\em splitting}. So given a choice of $\nabla \in \bp$, $\mathcal{E}_A$ decomposes as the direct sum
\[
{\mathcal{E}}_{A} \overset{\nabla}{\cong}  \mathcal{E}  (1) \oplus {\mathcal{E}} _{b} (1).
\]

The {\em standard tractor bundle}, $\mathcal{E}_{A}$, is the dual bundle to the standard cotractor bundle and so has the composition series
\[
0 \rightarrow \mathcal{E}(-1) \xrightarrow{X^A} \mathcal{E}^A \xrightarrow{Z^b_A} \mathcal{E}^b(-1) \rightarrow 0
\] 
where $X^A$ and $Z^b_A$ are projectively invariant. Given a choice of splitting, we denote the lifting map from the weighted tangent bundle to tractor bundle by $W_a^A:\mathcal{E}^a(-1)\rightarrow \mathcal{E}^A$ and the projection by $Y_A:\mathcal{E}^A\rightarrow \mathcal{E}(-1)$. By definition these satisfy the following relations: 

$\ $

\begin{center}
\begin{tabular}{ c| c  c }
 & $X^A$ & $W^A_a$ \\
\hline
$Y_A$ & 1 & 0 \\ 

$Z_A^b$ & 0 & $\delta^b_a$ \\ 
\end{tabular}
\end{center}

$\ $

We denote sections $V^A\in \Gamma(\mathcal{E}^A)$ and $U_A\in \Gamma(\mathcal{E}_A)$, respectively, by $V^A=W^A_a\nu^a + X^A\rho$ and $U_A = Y_A\xi + Z_A^a\mu_a$. In the presence of a splitting we will often abuse notation and denote these sections as follows,
\begin{align*}
V^A = W^A_a\nu^a + X^A\rho \overset{\nabla}{=} \left(
\begin{array}{c}
\nu ^{a}  \\
\rho  \\
\end{array} 
\right)
\ \ \ \ \ \operatorname{and} \ \ \ \ \ 
U_A = Y_A\xi + Z_A^a\mu_a \overset{\nabla}{=} \left(
\begin{array}{c}
\xi   \\
\mu_{a}  \\
\end{array} 
\right).
\end{align*}
When we wish to suppress the abstract indices we will denote the tractor and cotractor bundles by $\mathcal{T}$ and $\mathcal{T}^*$. respectively.

Associated with a projective structure on an $n$-dimensional manifold
$M$ is a canonically determined linear connection
$\nabla^{\mathcal{T}}$, on the bundle $\mathcal{T}$, known as the {\em
  normal tractor connection}.  In terms of a splitting the tractor and
cotractor connection this is given explicitly by
\begin{align}
\nabla ^{\mathcal{T}}_{a }
\left(
\begin{array}{c}
\nu ^{b}  \\
\rho  \\
\end{array} 
\right)
= 
\left(
\begin{array}{c}
\nabla_{a} \nu ^{b} + \rho \delta ^{b}_{a}  \\
\nabla _{a} \rho - P_{ab} \nu ^{b}  \\
\end{array} 
\right)
\ \ \ \ \ \text{\normalfont and} \ \ \ \ \ 
\nabla ^{\mathcal{T}_{*}}_{a} 
\left(
\begin{array}{c}
\xi  \\
\mu _{b}  \\
\end{array} 
\right)
 = 
 \left(
\begin{array}{c}
\nabla _{a} \xi - \mu _{a}  \\
\nabla _{a} \mu _{b} + P_{ab} \xi  \\
\end{array} 
\right),
\end{align} 

\noindent where $P_{ab}$ denotes the projective Schouten tensor as
defined in \cite{BEG} and \cite{CSlov}. We shall be mainly interested
in affine connections $\nabla$ that are {\em special}, meaning that
$\nabla$ preserves a volume density. Then, with the curvature
$R_{ab}{}^c{}_d$ of $\nabla $ on $TM$ given by
$(\nabla_a\nabla_b-\nabla_a\nabla_b)v^c=R_{ab}{}^c{}_dv^d$, we have
$(n-1)P_{bd}=\operatorname{Ric}_{bd}$, where $\operatorname{Ric}$ is the
Ricci tensor $R_{ab}{}^a{}_d$.

\begin{remark} \label{2.1}
The normal tractor connection is equivalent to the normal Cartan
connection if we view our projective manifold as a Cartan geometry
$(\mathcal{G},\omega)$ of type $(G ,P)$ where $G:=SL(\mathbb{R}^{n+1})$ and $P$ is
the parabolic subgroup stabilizing a fixed ray in
$\mathbb{R}^{n+1}$. Then for any $G$ representation $V$ we say that
$\mathcal{V} \colonequals (\mathcal{G}\times _P G) \times _G V =
\mathcal{G} \times_P V$ is a tractor bundle \cite{CGH2}. The Cartan
connection $\omega$ extends to a G-principal connection on
$\mathcal{G}\times _P G$, which in turn, induces a linear connection
on $\mathcal{V}$ called the tractor connection. In this language the
standard tractor bundle corresponds to the standard representation of
$G\cong SL(n+1)$, i.e. $\mathcal{T} = \mathcal{G} \times_P \mathbb{R}^{n+1}$.
\end{remark}

We will also be using the projectively invariant Thomas $D-$operator
$D_A:\mathcal{E}^{\bullet}(\omega)\rightarrow
\mathcal{E}_A^{\bullet}(w-1)$, as in \cite{BEG}.  Here $\mathcal{E}^{\bullet}$ denotes
any tractor bundle, and in a splitting $D_A$ is defined by

\begin{align*}
D_A U^{\bullet} \colonequals \left(
\begin{array}{cc}
\omega U^{\bullet} \\
\nabla_a U^{\bullet}   \\
\end{array} 
\right)
 =Y_A \omega U^{\bullet} + Z^a_A \nabla_a U^{\bullet},
\end{align*}
where $\nabla$ denotes the connection that couples the affine
connection of the spliting with the tractor connection. We are giving
the operator both in terms of the matrix presentation and tractor
injectors. Note that in particular this acts on projective densities:
$D_A:\mathcal{E}(\omega)\rightarrow\mathcal{E}_A(\omega-1)$, again
given explicitly by $\sigma \mapsto Y_A\sigma + Z^a_A\nabla_a\sigma$
where now $\nabla $ is simply the affine connection associated with
the given splitting.

In a splitting, sections $H^{AB} \in \Gamma ({\mathcal{E}^{(AB)}})$ and $H_{AB} \in \Gamma ({\mathcal{E}_{(AB)}})$ can be expressed as follows,

\begin{align}
{H^{AB}} = 
\left(
\begin{array}{cc}
\zeta^{ab} \\
\lambda^{a}  \\
\tau  \\
\end{array} 
\right) \colonequals W^A_cW^B_d\zeta^{cd} + 2 X^{(B}W^{A)}_c\lambda^c +X^AX^B\tau
\ \ \ \ \ \text{\normalfont and} \ \ \ \ \ 
{H_{AB}} = 
\left(
\begin{array}{cc}
\tau  \\
\lambda _{a}\\
\zeta_{ab} \\
\end{array} 
\right).
\end{align}
These could also be given by square symmetric matrices, but we use the above ``column'' form for ease of readability.
For later reference we note that the tractor curvature of $H^{AB}\in \Gamma(\mathcal{E}^{(AB)})$ is given by
\begin{align}
\Omega\indices{_{ab}^{C}_{E}}H^{DE} = (\nabla_a\nabla_b - \nabla_b\nabla_a)
\left(
\begin{array}{c}
\ze ^{cd}  \\
\lambda^c  \\
\tau \\
\end{array} 
\right)
=
\left(
\begin{array}{c}
W\indices{_{ab}^c_e}\ze^{de} + W\indices{_{ab}^d_e}\ze^{ce}  \\
W\indices{_{ab}^c_d}\lambda^d - Y_{abd}\ze^{cd}  \\
-2Y_{abc}\lambda^c \\
\end{array} 
\right),
\end{align}
where $W\indices{_{ab}^c_d}$, the {projective Weyl tensor} is totally trace-free and $Y_{abc} \colonequals \nabla_aP_{bc}-\nabla_b P_{ac}$ is the {\em projective Cotton tensor}.

Let  $\epsilon^2_{a_1 \cdots a_nb_1 \cdots b_n}$ denote the canonical section of $\mathcal{E}_{[a_1 \cdots a_n][b_1\cdots b_n]}(2n+2)$ which gives the identifying bundle map $\mathcal{E}^{[a_1 \cdots a_n][b_1 \cdots b_n]} \rightarrow \mathcal{E}(2n+2)$. This allows us to define the determinant of weighted contravariant 2-tensors as follows 
\begin{align}
\bdet : \mathcal{E}^{ab}(m) & \rightarrow \mathcal{E}(nm + 2n +2) \\
\nonumber \sigma^{ab} & \mapsto \epsilon^2_{a_1 \cdots a_nb_1 \cdots b_n}\sigma^{a_1b_1}\cdots \sigma^{a_nb_n}.
\end{align}
Next the projectively invariant parallel tractor
$$
\epsilon^2_{AB\cdots
  CDE\cdots E} \colonequals \epsilon^2_{b \cdots ce \cdots
  f}Y_{[A}Z^b_B \cdots Z^c_{C]}Y_{[D}Z^e_E \cdots Z^f_{F]},$$ which is
the square of the {\em tractor volume form} in the orientable case,
allows us to take determinants of contravariant 2-tractors,
\begin{align*}
\det : \mathcal{E}^{AB} & \rightarrow \mathbb{R} \\
H^{AB} & \mapsto \epsilon^2_{A_0 \cdots A_nB_0 \cdots B_n}H^{A_0B_0}\cdots H^{A_nB_n}.
\end{align*}

\subsection{Projective compactification}\label{proj-c-sec}

Projective compactification is a notion of compactification for affine
connections that is connected to projective differential geometry. It
was introduced in \cite{CG1} following the observation of special
cases in \cite{CGH1,CGH2,FG}.  For pseudo-Riemannian metrics this is
defined in the first instance via the Levi-Civita connection. First we give some background.

Let $M$ be a manifold and $\Sigma$ a smoothly embedded submanifold of
codimension 1 which we will call a {\em hypersurface}. A {\em local
  defining function} for a hypersurface $\Sigma$ is a smooth function
$r:U \rightarrow \mathbb{R}_{\geq 0}$, defined on an open  subset $U$ of $M$,
satisfying $\mathcal{Z}(r) = \Sigma \cap U$ and $\mathcal{Z}(dr) \cap
\Sigma = \varnothing$ on $\Sigma \cap U$, where $\mathcal{Z}(-)$
denotes the zero locus. Then, extending this concept, a {\em defining
  density of weight w} is a local section $\sigma$ of $\mathcal{E}(w)$
such that $\sigma = r\hat{\sigma}$, where $r$ is a defining function
for $\Sigma$ and $\hat{\sigma}$ is a section of $\mathcal{E}(w)$ that
is nonvanishing on $U$. Phrased differently, $\sigma\in \Gamma
(\mathcal{E}(w))$ is a defining density of weight $w$ if it satisfies
$\mathcal{Z}(\sigma) = \Sigma\cap U$ and $\mathcal{Z}(\nabla \sigma)
\cap \Sigma = \varnothing$, for some, equivalently any, connection
$\nabla$ on $\mathcal{E}(w)$. If $\si$ with these properties is defined globally then
$M_0 \colonequals
\mathcal{Z}(\sigma)$ is a separating hypersurface in that it partitions
$M$ into the disjoint union
\begin{equation}\label{strat}
M= M_{-} \cup M_{0} \cup M_{+}
\end{equation}
of open components $M_{-} \colonequals \{x \in M : \sigma <0\}$ and $M_{+} \colonequals \{x \in M : \sigma >0\}$, and closed component $M_{0}$. 
The components $M_{-}$, $M_{0}$, and $ M_{+}$ are not
necessarily connected. Note $M \backslash M_{\pm}$ is a manifold with boundary $M_0$. 

On a manifold $\overline{M}$, with boundary $\partial M$ and interior $M$, a connection $\nabla$ on $TM$ is said to be {\em projectively compact of order} $\alpha \in \mathbb{R}_{+}$ if for any point $x\in \partial M$ there is a local defining function $r:U \rightarrow \mathbb{R}_{\geq 0}$ defined on an open subset $U\subseteq \overline{M}$ containing $x$ such that the projectively related connection 
\begin{align*}
\hat{\nabla}_{\mu}\xi = \nabla_{\mu}\xi +\frac{dr(\mu)}{\alpha r}\xi + \frac{dr(\xi)}{\alpha r}\mu,
\end{align*}
defined on $U\cap M$, is smooth up to the boundary for all vector fields $\mu$ and $\xi$ that are smooth up to the boundary.

Recal that the bundles $\cE(w)$ are oriented. For any $w\neq 0$, it is
well known, and easily verified, that any nowhere-vanishing section
$\sigma\in\cE(w)$ determines a connection $\nabla$ in $\bp$
characterised by $\nabla_a\sigma=0$. For $0\neq w \in\mathbb{R}$ we call
a nowhere vanishing section of $\cE(w)$ as well as its corresponding
connection $\nabla^\si$ in $\bp$ a \emph{scale}. Note $\nabla\in\bp$ is a scale
if and only if it is special in the sense that it preserves a volume density.  

We will use often a characterization of projective compactness from part (ii) of Proposition 2.3 in \cite{CG1}: 
\begin{proposition}\label{2.2} 
Let $\overline{M}$ be a smooth $n$--dimensional manifold with boundary
$\partial M$ and with interior $M$. Let $\alpha\in \mathbb{R}_+$. 
 Suppose that $\overline{M}$ is endowed with a projective structure, and
  that $\si\in\Ga(\Cal E(\al))$ is a defining density for $\partial
  M$. Then one can view $\si$ as a scale for the restriction of the
  projective structure to $M$ and the affine connection $\nabla^\si$ on
  $M$ determined by this scale is projectively compact of order $\al$.
\end{proposition}
\noindent Thus in the setting of \eqref{strat}  it follows that $(M\backslash
 M_{\mp},[\nabla^{\sigma}])$ is the order $\alpha$ projective
 compactification of $(M_{\pm},\nabla^{\sigma})$.
 For more on projective compactness,
 see \cite{CG2}.

$\ $

\section{BGG Equations and the Metrizability Equation} \label{BGGsec}
\setcounter{theorem}{0}

Now we give a brief overview of the BGG machinery of \cite{CD,CSS},
drawing from  the summaries in \cite{CGH2,CGM} the tools necessary for our
purposes. Given a tractor bundle $\mathcal{V}$, via its tractor
connection we form the exterior covariant derivative on
$\mathcal{V}$-valued forms to obtain the de Rham sequence twisted by
$\mathcal{V}$.
\[
0 \xrightarrow{} \mathcal{V} \xrightarrow{d^{\nabla}} \mathcal{V} \otimes \mathcal{E}_a \xrightarrow{d^{\nabla}} \mathcal{V} \otimes \mathcal{E}_{[ab]} \xrightarrow{d^{\nabla}} ...
\]
Then, via the canonical map
\[
\dagger : \mathcal{E}_a \rightarrow \operatorname{End}(\mathcal{T}) \ \ \ \ \operatorname{given\ by} \ \ \ \ \alpha_a \mapsto X^BZ_A^a \alpha_a
\]
one can construct a special case of the Kostant codifferential $\partial^{*}$, that gives a complex of natural bundle maps on $\mathcal{V}$-valued differential forms going in the opposite direction to the twisted de Rham sequence,
\[
0 \xleftarrow{\partial^{*}} \mathcal{V} \xleftarrow{\partial^{*}}  \mathcal{V} \otimes \mathcal{E}_a \xleftarrow{\partial^{*}}  \mathcal{V} \otimes \mathcal{E}_{[ab]} \xleftarrow{\partial^{*}}  ...
\]

The homology of this sequence gives natural subquotient bundles 
\begin{align*}
H_{k}(M,\mathcal{V})\colonequals \operatorname{ker}(\partial^{*})/\operatorname{im}(\partial^{*}).
\end{align*}
There are natural bundle projections $\Pi_{k}:
\operatorname{ker}(\partial^*) \subseteq \mathcal{V} \otimes
\mathcal{E}_{[ab...c]} \rightarrow H_{k}(M, \mathcal{V})$, from the
indicated $\mathcal{V}$-valued $k$-forms to the $k$th BGG homology.
Given a smooth section $\rho$ of $H_{k}(M, \mathcal{V})$ there is a
unique smooth section $L_{k}(\rho)$ of
$\operatorname{ker}(\partial^*)\subseteq \mathcal{V}\otimes
\mathcal{E}_{[a...b]}$ such that $\Pi_{k}(L_{k}(\rho))=\rho$ and
$\partial^{*}(d^{\nabla^{\mathcal{V}}} L_{k}(\rho))$$=0$.  This
characterizes a projectively invariant differential operator $L$
called the BGG splitting operator, or just the {\em splitting
  operator}. We can then define the $k$th BGG operator
$\Theta_{k}:H_{k}(M, \mathcal{V}) \rightarrow H_{k+1}(M, \mathcal{V})$
by $\rho \mapsto \Pi_{k+1}(d^{\nabla^{\mathcal{V}}}L_{k}(\rho))$. It
follows from these definitions that parallel sections of $\mathcal{V}$
are equivalent to (via $\Pi_0$ and $L_0$) a special class of so-called
{\em normal} solutions of the first BGG operator $\Theta_{0}:H_{0}(M,
\mathcal{V}) \rightarrow H_{1}(M, \mathcal{V})$ associated with
$\mathcal{V}$. Equations induced on the sections of $H_0(M,
\mathcal{V})$ by the BGG operator $\Theta_0$ are known as ({\em
  first}) {\em BGG equations}. Note that the BGG sequence, given by
the BGG operators, is not a complex in general, unless the connection
$\nabla^{\mathcal{V}}$ is flat.

We consider two related BGG equations, determined via application of the BGG machinery to $\mathcal{E}_{AB}$ and $\mathcal{E}^{AB}$, respectively. The second is the metrizability equation. 

\begin{proposition} \label{3.1} 
Let $(M, \bp)$ be a projective manifold. The first BGG operator $\Theta_0: H_0(M,\mathcal{E}_{(AB)}) \rightarrow H_1(M,\mathcal{E}_{(AB)})$, induces the following third order, totally symmetric, equation on $\tau \in \mathcal{E}(2)$,
\begin{align}
\nabla_{(a} \nabla _{b} \nabla_{c)} \tau + 4P_{(ab} \nabla _{c)} \tau + 2 \tau \nabla_{(a} P_{bc)} =0.
\end{align}
\end{proposition}

\begin{proof}
Given a section $H_{AB}\in \Gamma(\mathcal{E}_{(AB)})$ we begin by computing $\nabla_{c}^{\mathcal{T^{*}}} {H_{AB}}$.
\begin{align*}
\nabla_{c}^{\mathcal{T^{*}}} {H_{AB}} & = 
\nabla_{c}^{\mathcal{T^{*}}}
\left(
\begin{array}{cc}
\tau  \\
\lambda _{a}\\
\zeta_{ab} \\
\end{array} 
\right)
= \left(
\begin{array}{cc}
\nabla_{c} \tau - 2 \lambda _{c} \\
\nabla _{c} \lambda _{a} + P_{ca} \tau - \zeta_{ca}\\
\nabla_{c} \zeta_{ab} + 2P_{c(b} \lambda_{a)} \\
\end{array} 
\right). \\
\end{align*}

Then $\partial^*(\nabla^{\mathcal{T}^*}_cH_{AB}) = 2(X^DZ^c_{(A}\nabla_c^{\mathcal{T}^*}H_{B)D})$. To explicitly determine the splitting operator we set $\partial^*(\nabla^{ \mathcal{T}^*} H_{AB})=0$, which yields the following system of equations;
\begin {align*}
\zeta_{ab} & = \nabla _{a} \lambda _{b} + P_{ab} \tau , \\
\lambda _{a} & =  \frac {1}{2}\nabla _{a} \tau.  
\end{align*}
Thus a section  in the image of the splitting operator is of the form
\begin{align*}
{H_{AB}} = L(\tau)  = 
\left(
\begin{array}{cc}
\tau  \\
\frac{1}{2}\nabla_{a}\tau \\
\frac{1}{2}\nabla_{a}\nabla_{b}\tau + P_{ab}\tau \\
\end{array} 
\right).
\end{align*}
Note that the $Z^a_AZ^b_B$ component of $\nabla_c^{\mathcal{T}^*}L(\tau)$ is the only non-vanishing component. It is precisely, 
\begin{align*}
\frac{1}{2}\nabla _{a} \nabla_{b} \nabla_{c} \tau + \tau \nabla _{a} P_{bc} + P_{bc}\nabla _{a} \tau + \frac{1}{2} P_{ac} \nabla_{b} \tau + \frac{1}{2}P_{ab}\nabla _{c} \tau.
\end{align*}
Via the Kostant codifferential $\partial^*$ it is straightforward\footnote{For details see Sec. 3.1 of \cite{CG1}} to verify that $H_1(M,\mathcal{E}_{AB})=\mathcal{E}_{(ab)c}(2)/\mathcal{E}_{a[bc]}(2) \cong \mathcal{E}_{(abc)}(2)$. Thus, symmetrizing the expression above, gives 
\begin{align}
\Theta_0(\tau) \colonequals \Pi_1(d^{\nabla}L_0(\tau)) = \nabla_{(a} \nabla _{b} \nabla_{c)} \tau + 2 \tau \nabla_{(a} P_{bc)} + 4P_{(ab} \nabla _{c)} \tau.
\end{align}
\end{proof}

\begin{proposition} \label{3.2}
Let $(M, \bp)$ be a projective manifold. The first BGG operator $\Theta_0: H_0(M,\mathcal{E}^{(AB)}) \rightarrow H_1(M,\mathcal{E}^{(AB)})$, induces the following projectively invariant first order equation on $\mathcal{E}^{ab}(-2)$,
\begin{align}
\nabla_{c} \zeta^{ab} - \frac{1}{n+1}\delta_{c}^{a}\nabla_{d}\zeta^{db} - \frac{1}{n+1}\delta_{c}^{b}\nabla_{d}\zeta^{ad} = 0.
\end{align}
\end{proposition}

\begin{proof}
Let $H^{AB}\in\Gamma(\mathcal{E}^{(AB)})$. Then we compute $\nabla_{c}^{\mathcal{T}} {H^{AB}}$.
\begin{align*}
\nabla_{c}^{\mathcal{T}} {H^{AB}} & = 
\nabla_{c}^{\mathcal{T}}
\left(
\begin{array}{cc}
\zeta^{ab} \\
\lambda ^{a}\\
\rho  \\
\end{array} 
\right) 
= \left(
\begin{array}{cc}
\nabla_{c} \zeta^{ab} + 2\delta_{c}^{(a} \lambda ^{b)} \\
\nabla _{c} \lambda ^{a} + \delta ^{a}_{c}\rho - P_{cb}\zeta^{ab} \\
\nabla_{c} \rho - 2P_{ca} \lambda ^{a} \\
\end{array} 
\right). \\
\end{align*}
Then $\partial^*(\nabla_c^{\mathcal{T}}H^{AB})= Z_D^cX^{(A}\nabla_c^{\mathcal{T}}H^{B)D}=0$ gives the following system of equations.
\begin{align}
\nabla_{c} \zeta^{ab} & = - 2\delta_{c}^{(a} \lambda ^{b)}, \\
\nabla_{c} \lambda^{a} & = P_{cb} \zeta^{ab} - \delta^{a}_{c}\rho ,
\end{align}
Tracing gives

\begin{align}
\lambda^{a} & = \frac{-1}{n+1}\nabla_{b} \zeta^{ab}, \\
\rho & = \frac{1}{n} P_{ba} \zeta^{ab} + \frac{1}{n(n+1)} \nabla _{a} \nabla_{b} \zeta^{ab}.
\end{align}
It follows that a symmetric bilinear form, $H^{AB}$, on the cotractor bundle in the image of the splitting operator is of the form
\begin{align*}
{H^{AB}} = L(\zeta^{ab}) = 
\left(
\begin{array}{cc}
\zeta^{ab} \\
\frac{-1}{n+1}\nabla_{b} \zeta^{ab}  \\
\frac{1}{n} P_{ba} \zeta^{ab} + \frac{1}{n(n+1)} \nabla _{a} \nabla_{b} \zeta^{ab}  \\
\end{array} 
\right).
\end{align*}
Substituting gives the following first-order BGG equation on $\mathring{\mathcal{E}}_c^{(ab)} (-2)$ (which denotes the trace-free component of $\mathcal{E}_c^{(ab)} (-2)$)

\begin{align}
\operatorname{trace-free}(\nabla_c\ze^{ab}) = 0 \Longleftrightarrow  \nabla_{c} \zeta^{ab} - \frac{1}{n+1}\delta_{c}^{a}\nabla_{d}\zeta^{db} - \frac{1}{n+1}\delta_{c}^{b}\nabla_{d}\zeta^{ad}=0.
\end{align}
Projective invariance follows from a straightforward computation. It is easy to see that $H_1(M, \mathcal{E}^{AB}) = \mathring{\mathcal{E}}^{(ab)}_c(-2)$. So we have given the explicit form of $\Theta_0(\zeta^{ab})\colonequals \Pi_1(d^{\nabla}L(\zeta^{ab}))=0$, which is  the metrizability equation (\ref{meteq}). Thus the metrizability equation of  Mikes and Sinjukov is seen to be a first BGG equation.  
\end{proof}
In summary, we have the following:

\begin{corollary} \label{3.3}
Let $\tau \in \Gamma (\mathcal{E}(2))$ and $\ze \in \Gamma (\mathcal{E}^{ab}(-2))$. Then their images under their respective splitting operators, both denoted by $L$, are given by
\begin{align*}
L(\tau)  = 
\left(
\begin{array}{cc}
\tau  \\
\frac{1}{2}\nabla_{a}\tau \\
\frac{1}{2}\nabla_{a}\nabla_{b}\tau + P_{ab}\tau \\
\end{array} 
\right) 
\ \ \ \text{\normalfont and} \ \ \ 
L(\zeta^{ab}) = 
\left(
\begin{array}{cc}
\zeta^{ab} \\
\frac{-1}{n+1}\nabla_{b} \zeta^{ab}  \\
\frac{1}{n} P_{ba} \zeta^{ab} + \frac{1}{n(n+1)} \nabla _{a} \nabla_{b} \zeta^{ab}  \\
\end{array} 
\right).
\end{align*}
\end{corollary} 

Note that a parallel section of a tractor bundle is necessarily in the image of the splitting operator. 

\section{Submanifolds and Stratifications} \label{results}
\setcounter{theorem}{0}

Recall we denote the degeneracy loci and the zero loci of tensors and densities, respectively, by $\mathcal{D}(-)$ and $\mathcal{Z}(-)$. Let $\mathcal{R}$ denote the map taking a tensor to its rank. We first consider a very simple case that is related to our study of degenerate solutions to the metrizability equation. 

\begin{lemma} \label{4.1}
Let $(M, \bp)$ be a projective manifold equipped with $\tau \in \Gamma(\mathcal{E}(2))$ such that the (possibly degenerate)
symmetric bilinear form, $L(\tau)$, on the tractor bundle, of signature $(p,q,r)$, satisfies 
$\mathcal{R}(\nabla\nabla\tau) < \mathcal{R} (L(\tau ))$ on 
$\mathcal{Z}(\tau)$. Then the following hold:

$(i)$ Either $\tau$ is nowhere zero or $L(\tau)$ is not definite and its 
zero locus is a smoothly embedded separating hypersurface $M_0$. If 
$L(\tau)$ has signature $(p,q,r)$ then $M$
is stratified by the strict sign of the 2-density $\tau$ and is partitioned as 
$$
M=  \coprod\limits_{i \in \{+,0,-\}} M_i
$$ with $\tau >0$ on $M_+$ and $\tau <0$ on $M_-$, and 
$\tau = 0$ on $M_0$. 
The components $M_+$, $M_0$, and $ M_{-}$ are not
necessarily each connected. 

$(ii)$ If $r=0$ and $M_0 \neq \varnothing$ then $M_0$ inherits a conformal structure $(M_0, \bc)$ of signature $(p-1,q-1)$. 

$(iii)$ If $M$ is closed, then $(M\backslash M_{\mp},\bp)$ is a projective compactification of order 2 of $(M_{\pm},\nabla^{\tau})$, with boundary $M_0$, where $\nabla^{\tau}\in \bp$ is the connection that preserves $\tau$ away from $\mathcal{Z}(\tau)$.\\
\end{lemma}

\begin{proof}
First we will show that $\mathcal{Z}(\tau) \cap \mathcal{Z}(\nabla_{a}\tau) = \varnothing$. Suppose, for contradiction, that $\mathcal{Z}(\tau) \cap \mathcal{Z}(\nabla_{a}\tau) \neq \varnothing$. Then, using the formula for $L(\tau)$, at $x \in \mathcal{Z}(\tau) \cap \mathcal{Z}(\nabla_{a}\tau)$, $H_{AB}$ reduces to 
\begin{align*}
{L(\tau)} & = 
\left(
\begin{array}{cc}
0 \\
0  \\
\frac{1}{2}\nabla_{a}\nabla_{b}\tau \\
\end{array} 
\right),
\end{align*}
giving us that $\mathcal{R}(\nabla_{a}\nabla_{b}\tau) = \mathcal{R}(L(\tau))$, a contradiction. 

Now, choosing a nonvanishing $\gamma \in \Gamma(\mathcal{E}(1))$. Then $\gamma ^{-2} \tau$ is a defining function for $M_0$, whence it follows from the implicit function theorem that $\mathcal{Z}(\tau)$ is a smoothly embedded submanifold of codimension $1$, i.e. a smoothly embedded hypersurface. It is clearly separating, since $\nabla\tau \neq 0$ along $\mathcal{Z}(\tau)$, so $M$ decomposes as the disjoint union of $M_+$, $M_0$, and $M_-$.

Thus $\tau$ is a defining density of weight 2 for $M_0$, so the claim
of projective compactness of $\nabla^\tau$ follows at once from
Proposition \ref{2.2}.  The (possibly degenerate) conformal
structure on the closed component $M_0$ follows by the same argument
as in Theorem \ref{3.2} of \cite{CGH1}.
\end{proof}

Our main application of Lemma \ref{4.1} is the following immediate consequence.

\begin{proposition} \label{4.2}
Let $(M, \bp)$ be a projective manifold equipped with $\tau\in \mathcal{E}(2)$ such that $L(\tau)$ is nondegenerate. Then $\mathcal{Z}(\tau)$ is either empty or it consists of smoothly embedded hypersurfaces of $M$. 
\end{proposition}

This generalizes a result from \cite{CG2} where $(M,\bp)$ was shown to
decompose according to the strict sign of $\tau$, as in Proposition
\ref{4.2}, when $L(\tau)$ was assumed to be nondegenerate {\em and parallel}. Our
proposition here drops the parallel assumption, instead needing only the
nondegeneracy of $L(\tau)$.

A projective manifold equipped with a nondegenerate symmetric bilinear
form, $L(\tau)$, on the tractor bundle has a canonical
pseudo-Riemannian structure $(M_{\pm},g)$ on the open orbits $M_{\pm}$
where $g_{ab}\colonequals P_{ab}$. The projective Schouten is seen to
be nondegenerate since $L(\tau) = (\tau, 0,P_{ab}\tau)^t$ when working
in the scale $\nabla^{\tau}$ preserving the density $\tau$. It is
evident, by working in the splitting $\nabla^{\tau}$ and applying the
tractor connection, that if the metric $g$ is preserved by
$\nabla^{\tau}$, whence $\nabla^{\tau}$ is the Levi-Civita connection
for $g$, then $L(\tau)$ is in fact parallel. But, in general, the
Levi-Civita connection corresponding to this metric need not lie in
the projective class $\bp$ (however the structure can still be of considerable
interest, see e.g.\ \cite{Mar}).

\subsection{Degenerate solutions of the metrizability equation: the order 2 projective compactification case}\label{order2dsec}

Next we will consider the case where we are given a solution $\ze$ to the metrizability equation and hence a symmetric bilinear form on the cotractor bundle, given by $H^{AB}=L(\zeta^{ab})\in \Gamma(\mathcal{E}^{(AB)})$. In this subsection we will address the case where $L(\zeta^{ab})$ is everywhere nondegenerate. We will see that, for a solution $\ze$ of the metrizability equation, nondegeneracy of $L(\ze)$ is enough to imply that the degeneracy locus of $\zeta^{ab}$, when nonempty, is a smoothly embedded hypersurface. Where it exists, we let $\Phi_{AB} \colonequals (H^{AB})^{-1}$ denote the pointwise inverse of $H^{AB}$. Given a splitting, say $\nabla\in \bp$, $H$ and $\Phi$ can be written

\begin{align}
H^{AB}  = 
\left(
\begin{array}{cc}
\ze^{ab}  \\
\lambda^a \\
\rho \\
\end{array} 
\right) 
\ \ \ \text{\normalfont and} \ \ \ 
{\Phi_{AB}} = 
\left(
\begin{array}{cc}
\tau  \\
\eta _{a}\\
\phi_{ab} \\
\end{array} 
\right),
\end{align}
for smooth sections $\rho \in \Gamma (\mathcal{E}(-2))$, $\lambda^a
\in \Gamma (\mathcal{E}^a(-2))$, $\ze^{ab} \in \Gamma
(\mathcal{E}^{ab}(-2))$, $\tau \in \Gamma (\mathcal{E}(2))$, $\eta_a
\in \Gamma (\mathcal{E}_a(2))$, and $\phi_{ab} \in \Gamma
(\mathcal{E}_{ab}(2))$.

\begin{proposition} \label{4.3}
Let $(M, \bp)$ be a projective manifold equipped with a section $\ze^{ab}\in \Gamma(\mathcal{E}^{(ab)}(-2))$, such that $H^{AB}\colonequals L(\zeta^{ab})$ is everywhere nondegenerate. Suppose that the pointwise inverse of $H^{AB}$ is given by $\Phi_{AB}$, as above, in the splitting determined by a connection $\nabla \in \bp$. Then
\begin{align*}
\nabla_{i} \tau & = 2 \eta_{i} - 2\tau \eta_{c} \psi^{c}_{i} -\tau^{2}\omega_{i} - \chi^{cb}_{i}\eta_{c}\eta_{b}
\end{align*}
holds for smooth sections $\chi^{cb}_{i} \in \mathcal{E}^{cb}_{i}(-2)$, $\psi^{c}_{i} \in \mathcal{E}^{c}_{i}(-2)$, and $\omega_{i} \in \mathcal{E}_{i}(-2)$, where $\chi$ is given by
\begin{align*}
\chi^{cb}_{i} & = \nabla_{i} \zeta^{cb} - \frac{1}{n+1}\delta_{i}^{c}\nabla_{d}\zeta^{db} - \frac{1}{n+1}\delta_{i}^{b}\nabla_{d}\zeta^{cd}.
\end{align*}
\end{proposition}

\begin{proof}
By definition we have 
\begin{align*}
\Phi_{AC}H^{CB} = \delta^{B}_{A}.
\end{align*}
Applying $\nabla^{\mathcal{T}}_{i}$ to both sides gives
\begin{align*}
(\nabla^{\mathcal{T}}_{i} \Phi_{AC})H^{CB}= - \Phi_{AC} \nabla ^{\mathcal{T}}_{i} H^{CB}.
\end{align*}
Applying $\Phi_{BD}$ to each side gives
\begin{align*}
\nabla^{\mathcal{T}}_{i} \Phi_{AD} = -\Phi_{AC} (\nabla^{\mathcal{T}}_{i}H^{CB}) \Phi_{BD} .
\end{align*}
By the formula for the tractor connection 
\begin{align*}
\nabla^{\mathcal{T}}_{i} \Phi_{AD}
= \left(
\begin{array}{cc}
\nabla_{i} \tau - 2 \eta _{i} \\
\nabla _{i} \eta _{a} + P_{ia} \tau - \phi_{ia}\\
\nabla_{i} \phi_{ad} + 2P_{i(d} \eta_{a)} \\
\end{array} 
\right)
\end{align*}
and 
\begin{align*}
 \nabla^{\mathcal{T}}_{i}H^{CB} = 
\left(
\begin{array}{cc}
\nabla_{i} \zeta^{cb} + 2\delta_{i}^{(c} \lambda ^{b)} \\
\nabla _{i} \lambda ^{c} + \delta ^{c}_{i}\rho - P_{ib}\zeta^{cb} \\
\nabla_{i} \rho - 2P_{ic} \lambda ^{c} \\
\end{array} 
\right)
\equalscolon \left(
\begin{array}{cc}
\chi_{i}^{cb} \\
\psi_{i}^{c}\\
\omega_{i} \\
\end{array} 
\right).
\end{align*}
Then we compute
\begin{align*}
\nabla^{\mathcal{T}}_{i} \Phi_{AD} = -\Phi_{AC} (\nabla^{\mathcal{T}}_{i}H^{CB}) \Phi_{BD} 
\end{align*}
to get the following system of equations
\begin{align*}
\nabla_{i} \phi_{ad} & = - 2P_{i(d} \eta_{a)} - \phi_{ac}\chi^{cb}_{i}\phi_{bd} - \phi_{ac}\psi^{c}_{i}\eta_{d} - \eta_{a}\psi_{i}^{b}\phi_{bd} - \eta_{a}\omega_{i}\eta_{d}  \\
\nabla _{i} \eta _{a} & = \phi_{ia} - P_{ia} \tau  -\phi_{ac}\chi^{cb}_{i}\eta_{b} -  \phi_{ac}\psi^{c}_{i}\tau - \eta_{a}\psi^{b}_{i}\eta_{b} - \eta_{a}\omega_{i}\tau \\
\nabla_{i} \tau & = 2 \eta_{i} - 2\tau \eta_{c} \psi^{c}_{i} -\tau^{2}\omega_{i} - \chi^{cb}_{i}\eta_{c}\eta_{b}.
\end{align*}
The last equation is precisely what we set out to show.
\end{proof}

Note that $\chi^{cb}_{i}=0$ is just the metrizability equation (\ref{meteq}). Thus, for a solution $\zeta^{ab}$ of that equation, along $\mathcal{Z}(\tau)$ we have
\begin{align} \label{dkey}
\nabla_{i} \tau & = 2 \eta_{i}.
\end{align}

\begin{lemma}\label{4.4}
Let $(M, \bp)$ be an $n-$dimensional projective manifold equipped with $H^{AB}\in \Gamma (\mathcal{E}^{(AB)})$, a nondegenerate symmetric bilinear form on the cotractor bundle, with pointwise inverse $\Phi_{AB}$. Then $\mathcal{D}(\zeta^{ab}) = \mathcal{Z}(\tau)$ where $\tau = X^AX^B\Phi_{AB}$ and $\zeta^{ab} = Z_A^aZ_B^bH^{AB}$.
\end{lemma}

\begin{proof}
Let $\operatorname{Adj}(H)$ denote the tractor field that is given by the adjugate of $H$ in a local frame. In such a frame this is just the cofactor transpose. This has the property that $\operatorname{Adj}(H)_{AB}H^{AB}=\operatorname{det} (H)$, we compute
\begin{align*}
\text{\normalfont \bdet}(\zeta^{a_nb_n}) & =  \epsilon^2_{a_0\cdots a_{n-1}b_0\cdots b_{n-1}}Z^{a_0}_{A_0} \cdots Z^{a_{n-1}}_{A_{n-1}}Z^{b_0}_{B_0} \cdots Z^{b_{n-1}}_{B_{n-1}}H^{A_0B_0}\cdots H^{A_{n-1}B_{n-1}}\\
& =  \thorn X^{A_n}X^{B_n}\mathcal\epsilon^2_{A_0\cdots A_nB_0\cdots B_n}H^{A_0B_0}\cdots H^{A_{n-1}B_{n-1}}\\
& =  \thorn X^{A_n}X^{B_n} \text{\normalfont Adj}(H)_{A_nB_n}\\
& = \thorn \operatorname{det}(H)\tau,
\end{align*}
and so 
\begin{equation}\label{star}
\text{\normalfont \bdet}(\zeta^{a_nb_n}) = \thorn \operatorname{det}(H)\tau.
  \end{equation}
for some non-zero constant $\thorn$. The nondegeneracy of $H^{AB}$ allows us to conclude that $\mathcal{Z}(\tau) = \mathcal{D}(\zeta^{ab})$. 
\end{proof}

\noindent Next we show that the equation  
\begin{align}\label{EMconn}
\nabla_{c}H^{AB} + \frac{2}{n}X^{(A}W\indices{_{cE}^{B)}_F}H^{EF}=0,
\end{align}
is equivalent to the prolonged system in \cite{CGM} corresponding to the metrizability equation, where $W\indices{_{cE}^{B}_F} \colonequals Z_{E}^{e}\Omega\indices{_{ce}^{B}_F}$, and $\Omega$ denotes the tractor curvature. 

Notably, a solution $H^{AB}$ of \eqref{EMconn} is equivalent to a solution $\zeta^{ab}$ to the metrizability equation. This is straightforward to check, we just compute the slots and see that they agree with \cite{CGM, EM}\footnote{Eastwood and Matveev \cite{EM} use a slightly different convention for the tractor connection and projective Cotton tensor.}. Let us begin by writing
\begin{align*}
H^{DE}=X^DX^E\rho + 2 X^{(D}W^{E)}_e\mu^e+W^D_dW^E_e\ze^{de}.
\end{align*}
Then 
\begin{align*}
H^{DE}\Omega\indices{_{ab}^C_E} = & X^DX^C(-2Y_{abc}\mu^c) + X^{D}W^{C}_g(W\indices{_{ab}^g_d}\mu^d - Y_{abd}\ze^{cd}) \\
& +X^{C}W^{D}_d(W\indices{_{ab}^d_c}\mu^c - Y_{abc}\ze^{cd})+W^D_dW^C_g(W\indices{_{ab}^g_e}\ze^{de}+W\indices{_{ab}^d_e}\ze^{ge}).
\end{align*}
So,
\begin{align*}
H^{DE}X^{(F}\Omega\indices{_{ab}^{C)}_E}= & X^DX^{(F}X^{C)}(-2Y_{abc}\mu^c) + X^{D}X^{(F}W^{C)}_g(W\indices{_{ab}^g_d}\mu^d - Y_{abd}\ze^{cd}) \\
& + X^{(F}X^{C)}W^{D}_d(W\indices{_{ab}^d_c}\mu^c - Y_{abc}\ze^{cd})+W^D_dX^{(F}W^{C)}_g(W\indices{_{ab}^g_e}\ze^{de}+W\indices{_{ab}^d_e}\ze^{ge}).
\end{align*}
Which gives us that
\begin{align*}
Z^b_DH^{DE}X^{(F}\Omega\indices{_{ab}^{C)}_E}= & X^{(F}X^{C)}\delta^b_d(W\indices{_{ab}^d_c}\mu^c - Y_{abc}\ze^{cd})+\delta^b_dX^{(F}W^{C)}_g(W\indices{_{ab}^g_e}\ze^{de}+W\indices{_{ab}^d_e}\ze^{ge})\\
& + X^{(F}X^{C)}(-2Y_{abc}\ze^{cb})+X^{(F}W^{C)}_g(W\indices{_{ad}^g_e}\ze^{de}).\\
\end{align*}
Thus we see that the slots indeed agree with \cite{CGM}, i.e. 

\begin{align*}
\nabla_c 
\left(
\begin{array}{cc}
\zeta^{ab} \\
\mu^{a}  \\
\rho  \\
\end{array} 
\right)
= \frac{1}{n}
\left(
\begin{array}{cc}
0 \\
-W\indices{_{cd}^a_e}\ze^{de}  \\
2Y_{cba}\ze^{ba}  \\
\end{array} 
\right)
\end{align*}

\noindent We are now prepared to prove our first main result, Theorem \ref{one}, wherein we generalize a result in \cite{CGH2} showing that an everywhere nondegenerate parallel symmetric bilinear form on the standard tractor bundle induces a decomposition of the underlying manifold. Note that in the following Lemma we do not assume that $\zeta$ is a solution to the metrizability equation. Recall the notation of Proposition \ref{4.3}.

\begin{lemma} \label{4.5}
Let $(M,\bp)$ be a projective manifold equipped section $\zeta^{ab}\in\Gamma(\mathcal{E}^{(ab)})$ such that $H^{AB}=L(\zeta^{ab})$ is everywhere nondegenerate as a symmetric bilinear form on the cotractor bundle. Denote its pointwise inverse by $\Phi_{AB}$, as above. If $\mathcal{Z} (\tau) \cap \mathcal{Z} (\nabla_{i} \tau) \subseteq \mathcal{Z}(\chi_{i}^{ab}\eta _{a}\eta_{b})$, where $\chi$ is given in Proposition \ref{4.3}, then either
$\ze$ is everywhere non-degenerate or its degeneracy locus is a
smoothly embedded separating hypersurface
$M_0$.  If $L(\ze)$ is definite then the degeneracy locus
is empty, otherwise if $L(\ze)$ has signature $(p+1,q+1)$ then: \\
(i) $M$ is stratified by the strict signature of $\ze$ as a (density
weighted) bilinear form on $T^*M$ and the partitioning is  
$$
M=  \coprod\limits_{i \in \{+,0,-\}} M_{i}
$$ 
where $\zeta$ has signature $(p+1,q)$,
$(p,q+1)$,and $(p,q,1)$ on $M_{+}$, $M_{-}$, and $M_0$, respectively. \\
(ii) If $M$ is closed, then $(M\backslash M_{\mp},\bp)$ is an order 2 projective compactification of $(M_{\pm},\nabla^{\tau})$, with boundary $M_0$.\\
\end{lemma}

\begin{proof}
Let $x \in \mathcal{Z} (\tau) \cap \mathcal{Z} (\nabla_{i} \tau)$, if non-empty. It follows that $x \in \mathcal{Z} (\chi_{i}^{ab}\eta _{a}\eta_{b})$. Thus, at $x$, 
\begin{align*}
\nabla_{i} \tau & = 2 \eta_{i} - 2\tau \eta_{c} \psi^{c}_{i} -\tau^{2}\omega_{i} - \chi^{cb}_{i}\eta_{c}\eta_{b}
\end{align*}
from Proposition \ref{4.3} reduces to 
\begin{align*}
\eta_{i} = 0.
\end{align*}
This implies that $\Phi$ is degenerate, a contradiction. Hence $\mathcal{Z} (\tau) \cap \mathcal{Z} (\nabla_{i} \tau) =\varnothing$. Thus, by the implicit function theorem, $\mathcal{Z} (\tau)$ is a smoothly embedded hypersurface. By the previous Lemma $\mathcal{Z} (\tau) = \mathcal {D} (\zeta^{ab})$, whereby we conclude that $\mathcal{D} (\zeta^{ab})$ is a smoothly embedded hypersurface with defining density $\tau$ of weight 2, from which the order 2 projective compactness of the open components $(M_\pm,\nabla^{\tau}$) follows as in Lemma \ref{4.1}. Recall from Lemma \ref{4.4} that $\tau$ is equal to the determinant of $\zeta$ up to multiplication by a smooth nonvanishing function. Thus the manifold decomposes into the claimed disjoint union according to the strict sign of $\tau$ which itself varies according to the strict signature of $\zeta$. 
\end{proof}
The components $M_+$, $M_0$, and $ M_{-}$ in the preceding theorem are not
necessarily each connected. 

\begin{theorem} \label{4.6}
Under the conditions of the previous Lemma, with also  $\zeta^{ab}$ assumed to be a solution to the metrizability equation (so $\chi _{i}^{ab} = 0$),
 the following hold: 

 (i) $\mathcal{D}(\zeta^{ab})$, if nonempty, is a smoothly embedded hypersurface of $M$, and $M$ decomposes according to the previous Lemma \ref{4.5}. 
 
(ii) On $M_\pm$, $\zeta$ induces a pseudo-Riemannian metric $g_{\pm}$
 of the same signature as $\zeta$, where $g_{\pm}^{ab} =
 \operatorname{sgn}(\hat{\tau}) \hat{\tau} \zeta^{ab}|_{M_{\pm}}$,
 where $\hat{\tau}\colonequals \bdet(\ze)$. This satisfies $\nabla^{g_{\pm}}\in \bp$ and
 $ \hat{\tau}$ is a projective weight 2 defining density for $\mathcal{D}(\zeta^{ab})$.
 If $M$ is closed, then
 $(M\backslash M_{\mp},\bp)$ is an order 2 projective compactification
 of $(M_{\pm},\nabla^{g_\pm})$, with boundary $M_0$, where $\nabla^{g_\pm}$ is
 the Levi-Civita connection corresponding to $g_\pm$.

 (iii) $M_0$
 inherits a conformal structure with signature $(p,q)$. 
\end{theorem}

 \begin{proof}
 $(i)$ $\ze^{ab}$ is a solution to the metrizability equation so $\chi_i^{ab}=0$, whence the result follows from Lemma \ref{4.5}. 
 
   \noindent $(ii)$ The first statement expresses the standard
   relation between a solution of \eqref{meteq} and a metric $g_\pm$ such
   that $\nabla^{g_\pm}\in \bp$. Via the formula $g_{\pm}^{ab} =
 \operatorname{sgn}(\hat{\tau}) \hat{\tau} \zeta^{ab}|_{M_{\pm}}$ it follows that $\nabla^{g_\pm}$ also preserves $\zeta$ and $\hat{\tau}$, and that either of the latter characterize $\nabla^{g_\pm}\in \bp$. 
  Next that the weight 2 $ \hat{\tau}$ is
 a defining density for $\mathcal{D}(\zeta^{ab})$ is an immediate
 consequence of \eqref{dkey} and \eqref{star}. Since
 $\nabla^{g_\pm}\hat{\tau}=0$ and $ \hat{\tau}$ is a defining density
 for the smooth hypersurface $\mathcal{D}(\zeta^{ab}=:M_0$ it follows
 that $(M\backslash M_{\mp},\bp)$ is an order 2 projective
 compactification as claimed.

\noindent $(iii)$ $M_0$ inheriting a conformal structure follows by showing that $\ze$ has rank $k \geq n-1$ and then that the normal, $\nabla\tau$, to the hypersurface, $M_0$, is in the kernel of $\zeta$ so $\ze$ is nondegenerate on $T^*M_0$ and hence gives a conformal metric. 
 
 Given an $(n+1) \times (n+1)$ matrix of rank $m$, removing a column (resp. row) reduces the rank by at most 1. Then removing a row (resp. column) again decreases the rank by at most one. Hence and $n\times n$ submatrix has rank $k \geq m-2$. Picking a local frame so that $\ze$ and $L(\ze)$ can be presented as a matrices we see that we are in the case where $k \geq n-1$. Hence $\mathcal{R}({\ze}) \geq n-1$, so the kernel of $\ze$ has dimension less than or equal to $1$. In particular, when $\ze$ becomes degenerate its kernel has rank 1. Now we show that when $\ze$ degenerates that its kernel is spanned by $\nabla \hat\tau$. 
 
The adjugate of $\ze$ is symmetric and on $\mathcal{D}(\ze)$ it has
rank $1$, hence locally it can be written
$\operatorname{adj}(\ze)_{ab}=\alpha_a\alpha_b$ for some $1-$form
$\alpha$. Observe that, up to a nonzero constant,
$\bdet(\ze)=\operatorname{adj}(\ze)_{ab}\ze^{ab}$. So
$\alpha_a\ze^{ab}=0$ on $\mathcal{D}(\ze)$. Then on $\mathcal{D}(\ze)$
we have the following:
 \begin{align*}
\ze^{cd} \nabla_c\hat \tau & = \ze^{cd}\nabla_c(\epsilon^2_{a_1 \cdots a_nb_1 \cdots b_n}\ze^{a_1b_1}\cdots \ze^{a_nb_n}) \\
 & = n\ze^{cd}\epsilon^2_{a_1 \cdots a_nb_1 \cdots b_n}\ze^{a_1b_1}\cdots \ze^{a_{n-1}b_{n-1}}(\nabla_c \ze^{a_nb_n}) \\
 & = n\ze^{cd}\operatorname{adj}(\ze)_{a_nb_n}(\frac{1}{n+1}\delta^{a_n}u_c\nabla_e\ze^{eb_n}+\frac{1}{n+1}\delta^{b_n}_c\nabla_e\ze^{ea_n}) \\
 & = \frac{n}{n+1}( \ze^{a_nd} \alpha_{a_n}\alpha_{b_n} \nabla_e\ze^{eb_n}+\ze^{b_nd}\alpha_{a_n}\alpha_{b_n}\nabla_e\ze^{ea_n}) \\
 & = 0.
 \end{align*}
\end{proof}

Let $S\colonequals \operatorname{det}(L(\zeta^{ab}))$. In \cite{CG3} Proposition
3 it was shown that this defines a smooth function on $M$ that
generalizes the notion of scalar curvature. On each of $M_{\pm}$ it agrees (up to a constant factor) with the usual scalar curvature of
the metric $g_{\pm}$
(determined by $\zeta$ on $M_{\pm}$). But $S$ is well defined where the metric is singular (i.e. on $M_0$).  Note  that while
the scalar curvature can change sign on M, the {\em generalized scalar
curvature} $S$ obviously does not, given our assumptions.

\begin{remark} \label{4.7}
Let $(M,\bp)$ be a projective manifold satisfying the conditions of
Theorem \ref{4.6}, and assume that $\ze$ is a solution of the
metrizability equation. Assume $M_0 \neq \varnothing$, so
$\hat\tau\colonequals \bdet(\ze)$
is a defining density for $M_0$ regarded
as the boundary of the projective compactification of $(M_{\pm},
\nabla^{\tau})$. In \cite{CG3} the generalized scalar curvature, $S$,
was shown to be locally constant and nowhere vanishing on $M_0$. The
generalized scalar curvature being locally constant on $M_0$ is
equivalent to the contravariant index of a $\ze$ trace of the
projective Weyl tensor, $\xi^aW\indices{_{ab}^{c}_{d}}\zeta^{bd}$,
being tangential along $M_0=\mathcal{D}(\zeta^{bd})$, for all $\xi^a\in
TM_0$. One can see this as follows:


Since $S$ is locally constant on $M_0$ we know that along vector fields $\xi^c \in \Gamma (TM_0)$, the determinant of $H^{AB}$ has derivative zero. 
Note that, as seen above, $\tau \colonequals \Phi_{AB}X^AX^B$ is also a defining density for $M_0$. So, $\nabla \tau $ is nowhere zero along $M_0$, for  any scale $\nabla$  that is define in a neighborhood of  the boundary. So $\nabla\tau$ is a (weighted) conormal to $M_0$ and 
\[
D_A(\tau) =
\left(
\begin{array}{cc}
\tau \\
\nabla\tau
\end{array}
\right) =
\left(
\begin{array}{cc}
0 \\
\nabla \tau
\end{array}
\right)
\]
on $M_0$. Since the Thomas-D operator satisfies the Leibniz rule we see 
\begin{align*}
D_A \tau = (D_A\Phi_{BC})X^BX^C + 2\Phi_{AB}X^B.
\end{align*}
A straightforward computation shows that $(D_A \Phi_{BC})X^BX^C=0$, so that we have $D_A\tau = 2 \Phi_{AB}X^B$. Another straightforward computation gives that $\nabla_c^{\mathcal{T}} \operatorname{det}(H^{AB}) = (\nabla_a\tau) W\indices{_{cd}^a_e}\zeta^{de} + 2\tau Y_{cde}\zeta^{de}$. Of course, along the boundary this reduces to 
\[
\xi^c\nabla_c^{\mathcal{T}} \operatorname{det}(H^{AB}) = \xi^c(\nabla_a\tau) W\indices{_{cd}^a_e}\zeta^{de},
\]
where $\xi^c\in \Gamma(TM_0) \subset \Gamma(TM)$. If the contravariant index in the projective Weyl tensor is tangential then this is certainly zero. On the other hand, if it is equal to zero, then either the Weyl tensor vanishes in which case the contravariant index it tangential, or it does not vanish in which case it is forced to be tangential.

We will derive an analogous result, regarding the tangential index of the contracted Weyl, in the following subsection, where it will play a fundamental role in showing that the degeneracy locus is totally geodesic. 
\end{remark}

\subsection{Degenerate solutions of the metrizability equation: the order 1 projective compactification case}\label{dsec}

We begin with another simple case where we stipulate a fixed rank
condition on the splitting operator applied to a 2-density. Then we
will move to a ``dual case'' concerning solutions to the metrizability
equation which, under new assumptions, are related to scalar-flat
metrics.

\begin{proposition}\label{4.8}
Let $(M^n,\bp)$ be a connected, projective 
manifold equipped with a section $\tau \in \Gamma(\mathcal{E}(2))$ such that $\mathcal{R}(L(\tau))=1$. Then either $\mathcal{Z}(\tau)$ is empty or $M$
is locally stratified by the strict sign of a canonical 1-density $\si$, that is locally the square 
root of $\tau$ or $-\tau$. That is given $x\in M$ there is an open set $U\subseteq M$ containing $x$ such that $U$ is the disjoint union  
$$
U=  \coprod\limits_{i \in \{+,0,-\}} U_i
$$ with $\si >0$ on $U_+$ and $\si <0$ on $U_-$, and $\si = 0$ on $U_0$. If closed, $(U \backslash U_{\mp},\bp)$ are order 1 projective compactifications of $(U_{\pm},\nabla^{\si})$, with boundary $U_0$. In any case the latter inherits a projective structure. 

\end{proposition}

\begin{proof} 
Let $H_{AB} \colonequals L(\tau)$. The rank $1$ assumption implies
that, given any $x\in M$ there is an open neighborhood $U$ of $x$ such that
$H_{AB}=fV_AV_B$ for $f$ nowhere zero and $V_A$ also
nonvanishing. We work locally on such an open set.
If $f<0$, replace $\tau$ with its negative. So, without
loss of generality we can take $f= 1$. Then $\sigma \colonequals
X^AV_A = \sqrt{\tau}$. Clearly
$\mathcal{Z}(\sigma)=\mathcal{Z}(\tau)$. We show that $V_A=D_A\sigma$
on the closure of the set $\tilde{U}\colonequals \{x\in U : \sigma(x)
\neq 0\}$ and the hypersurface decomposition follows, as does the
order 1 projective compactification.

Away from $\mathcal{Z}(\tau)$ we can work in the scale $\nabla^{\tau}
\in \bp$ which preserves $\tau$. Then $\nabla^{\tau}$ preserves $\tau$
and $\sigma$ since $\tau= \sigma^2$. Thus it follows that, in the
scale $\nabla^{\tau}$, $D_A\sigma=Y_A \sigma$. Working, away from
$\mathcal{Z}(\tau)$, in the splitting $\nabla^{\tau}$ we see that
$H_{AB} = (\sigma ^2,0,P_{ab} \sigma^2)^t$. By our assumption that
this is rank 1, it follows that the projective Schouten vanishes away
from $\mathcal{Z}(\tau)$, and $V_AV_B=H_{AB}=Y_AY_B\tau$.
So on $U \backslash \mathcal{Z}(\tau)$ we have $H_{AB} =
(D_{A}\sigma)(D_{B}\sigma)$. In order to conclude that
$\mathcal{Z}(\sigma)$ is a smoothly embedded hypersurface we will show
that $H_{AB} = (D_{A}\sigma)(D_{B}\sigma)$ holds on
$\mathcal{Z}(\sigma)\cap U$. Then $D_A\sigma$ nonvanishing implies
that $\nabla_a\sigma \neq 0$ on $\mathcal{Z}(\sigma)$, and the result
follows from the implicit function theorem.

Now the cotractor $V$, in arbitrary scale, is of the form
\begin{align*}
V_{A} = 
\left(
\begin{array}{c}
\sigma \\
\mu_{a} \\
\end{array} 
\right).
\end{align*}
We will show that $X^{B}\nabla_{c}V_{B} = 0$. From the formula for the tractor connection this implies that $\mu_a=\nabla_a\sigma$, whence  $V_{A} = D_{A}\sigma$. Since the top two slots of $\nabla_cL(\tau)$ vanish, regardless of rank assumptions on $L(\tau)$, by Corollary 3.3, we see, in particular, that

\begin{align*}
0 & = X^{A}X^{B}\nabla_{c}H_{AB} =  2X^{A}X^{B}V_{A}(\nabla_cV_{B}) =  2\sigma X^{B}\nabla_cV_{B}.  \\
\end{align*}
Thus $\sigma X^{B}\nabla_cV_{B}=0$. Observe that if $\tilde{U} = \varnothing$ then $\mathcal{Z}(\sigma)$ is all of $U$, implying $L(\tau)$ has rank zero, a contradiction. Thus $\tilde{U}$ is nonempty. It follows, from the continuity of $\sigma$, that $\tilde{U}$ is open. Then on this nonempty open set we see that $X^{B}\nabla_cV_{B} = 0$. On the other hand, the zero locus $\mathcal{Z}(X^{B}\nabla_cV_{B})$ is a closed set, also by continuity. But the smallest closed set containing a given open set is the closure of that open set, hence cl$(\tilde{U}) \subseteq \mathcal{Z}(X^{B}\nabla_cV_{B})$, implying that on cl$(\tilde{U})$, 
\begin{align*}
V_{A} = 
\left(
\begin{array}{c}
\sigma \\
\nabla_{b}\sigma \\
\end{array} 
\right).
\end{align*}
But $V_A$ is nonvanishing on $U$, and in particular, on cl$(\tilde{U})$. Thus $\nabla_b\sigma \neq 0$ on the boundary of cl$(\tilde{U})$. Thus $\mathcal{Z}(\sigma)$ is is a smooth hypersurface. We can conclude that $V_A=D_A\sigma$ everywhere on the local region where $H_{AB}=V_AV_B$.

By Proposition 8 in \cite{CG2} we know that $\sigma P_{ab}$ admits a smooth extension to $U_0$, and by \cite{CG3} Proposition 3.1 $U_0$ is totally geodesic if and only if the smooth extension vanishes identically on $U_0$, where $U_0$ is given in the Proposition statement. But $P_{ab}$ vanishes identically on the open set $U \backslash U_0$ hence any smooth extension vanishes on this set as well as its closure (in $U$), $U$. We conclude that $U_0$ is totally geodesic, whence it inherits a projective structure via restriction of the ambient projective structure. Further, since the Schouten is symmetric it is a scalar multiple of the Ricci, the structure is in fact Ricci-flat.
If $\mathcal{Z}(\tau)$ is orientable then a consistent choice of sign for $\sigma$ can be made in which case the local sections $V_A$ can be patched to form a global nonvanishing section and $\mathcal{Z}(\sigma)$ is a separating hypersurface.
\end{proof}

Next we proceed to considering the case where
$\mathcal{R}(L(\zeta^{ab}))=n$.  Note that in this  case the generalized
scalar curvature $\operatorname{det}({L(\zeta^{ab})})$
obviously vanishes on all of $M$ since $L(\ze)$ is corank $1$. Thus trivially we
have the following:

\begin{proposition} \label{4.9}
  Let $(M^n, \bp)$ be a projective manifold equipped with a solution $\zeta$ of the metrizability equation such that $\mathcal{R}(L(\zeta))=n$. Then on $M \backslash \mathcal{D}(\zeta^{ab})$ the metric (inverse) $g^{ab} = \operatorname{sgn}(\tau)
\tau \zeta^{ab}$
  corresponding to $\zeta$  is scalar-flat.
\end{proposition}
Indeed considering $\zeta^{ab}$ away from its degeneracy
locus we have $g^{ab} = \operatorname{sgn}(\tau)
\tau \zeta^{ab}$, where $\tau =
\operatorname{\bdet}(\ze^{ab})$.
Letting $\nabla^{g} \in \bp$ denote
the corresponding Levi-Civita connection, and $R$ its scalar curvature
then in the scale determined by $g$ we have, as in \cite{CG3}
Proposition 3, that
\begin{align*}
L(\ze^{ab}) = 
\left(
\begin{array}{cc}
\operatorname{sgn}(\tau)\tau^{-1}g^{ab} \\
0 \\
\frac{1}{n}\operatorname{sgn}(\tau)\tau^{-1}R \\
\end{array} 
\right).
\end{align*} 
Since $g^{ab}$ is nondegenerate, and $L(\ze)$ is of corank $1$, it follows that $R=0$. Thus $M \backslash \mathcal{D}(\zeta^{ab})$ is scalar-flat.

Given a solution $\ze$ of the metrizability equation such that $\mathcal{R}(L(\zeta))=n$, its pointwise inverse is undefined. So, in order to study the degeneracy locus $\mathcal{D}(\zeta^{ab})$, we form the adjugate
$\overline{\mathcal{H}}_{AB}$, of $H^{AB}\colonequals L(\ze^{ab})$ by taking

\begin{align}\label{tradj}
\overline{\mathcal{H}}_{A_{0}B_{0}}:=\epsilon^2_{A_0A_1\cdots
A_nB_0B_1\cdots B_n} H^{A_{1}B_{1}}\cdots H^{A_nB_n}.
\end{align}

Since $H$ has rank $n$, we have that

\begin{align*}
H^{AB}\overline{\mathcal{H}}_{BC}= \frac{1}{n+1}\text{\normalfont det}(H^{AB})\delta^{A}_{C}=0.
\end{align*} 
$\overline{\mathcal{H}}_{AB}$ inherits its symmetry from $H^{AB}$, as
is clear from \eqref{tradj}, and is rank $1$ by the assumption that
$H^{AB}$ is rank $n$. As a symmetric rank $1$ tensor,
$\overline{\mathcal{H}}_{AB}$ is certainly locally simple. So  given
a point $x \in M$ there exists an open set $U \subseteq M$ containing
$x$ such that, on $U$, $\overline{\mathcal{H}}_{AB}= fI_{A}I_{B}$ for
some smooth function $f$ and cotractor $I$, both of which must be
nonvanishing since otherwise there would exist points in $M$ with
$\mathcal{R}(L(\ze))<n$, contradicting our assumption.  Then Lemma
\ref{4.4} {\em mutatis mutandis}, shows that
$X^AX^B\overline{\mathcal{H}}_{AB}=\bdet(\zeta^{ab})$. If $f>0$ define
$\mathcal{H}_{AB}\colonequals \overline{\mathcal{H}}_{AB}$ and if
$f<0$ define $\mathcal{H}_{AB}\colonequals
-\overline{\mathcal{H}}_{AB}$. We call $\mathcal{H}$ the {\em signed
  adjugate} of $H$. Then $\mathcal{H}$ is simple, symmetric, satisfies
$H^{AB}\mathcal{H}_{BC}=0$, and locally $\mathcal{H}_{AB}=fI_AI_B$ for
$f>0$.  So we can smoothly take a square root of $f$, and absorb
$\sqrt{f}$ into $I$. Thus without loss of generality we can take $f$
to be the constant function 1 and simply write $\mathcal{H}_{AB}=I_AI_B$,
locally, for a nonvanishing $I$.

\begin{proposition} \label{4.10}
Let $(M^n, \bp)$ be an projective manifold. Suppose that $H^{AB} \in \Gamma(\mathcal{E}^{(AB)})$ has rank $n$, with signed adjugate given locally by $\mathcal{H}_{AB}=I_AI_B$, $\zeta^{ab} \colonequals Z_{A}^{a}Z_{B}^{b}H^{AB}$, and $\sigma \colonequals X^{A}I_{A}$. Then the following hold locally: \\

$(i)$ $H^{AB}I_{A}=0.$ \\

$(ii)$ $\mathcal{D}(\zeta^{ab}) = \mathcal{Z}(\sigma)$.
\end{proposition}

\begin{proof}
$(i)$ This follows trivially from the fact that $\mathcal{H}_{AC}H^{AB}=0$. \\

$(ii)$ This is just a simple computation,
\begin{align*}
\mathcal{D}(\zeta^{ab}) & = \mathcal{D}(Z_{A}^{a}Z_{B}^{b}H^{AB}) = \mathcal{Z}(\operatorname{\bdet}(Z_{A}^{a}Z_{B}^{b}H^{AB})) =\mathcal{Z}(X^{A}X^{B}\mathcal{H}_{AB}) =\mathcal{Z}(X^AX^BI_AI_B)
\end{align*}
where the third equality follows from the proof of Lemma \ref{4.4}.
\end{proof}

\begin{theorem}\label{4.11}
  Let $(M,\bp)$ be a connected, projective 
manifold equipped with a solution $\ze$ of the metrizablility 
equation such that $L(\zeta)$ has signature $(p,q,1)$. If 
$\mathcal{D}(\zeta)=\varnothing$, then $\zeta$ induces a scalar-flat pseudo-Riemannian metric of signature $(p,q)$ on $M$. Otherwise, if $\varnothing \subsetneq \mathcal{D}(\ze)\subsetneq M$, then $\mathcal{D}(\ze)$ is a smoothly embedded hypersurface. If $M$ and $\mathcal{D}(\ze)$ are orientable then $\mathcal{D}(\ze)$ is a separating hypersurface and stratifies $M$ by the strict sign of a canonical 1-density, $\si$, that is locally a square 
root of $\bdet(\ze)$ or $-\bdet(\ze)$. The partition of $M$ is given by  
$$
M=  \coprod\limits_{i \in \{+,0,-\}} M_i
$$ with $\si >0$ on $M_+$ and $\si <0$ on $M_-$, and $\si = 0$ on $M_0$. \\
\end{theorem}

\begin{proof}
  Write $\mathcal{H}_{AB}$ for the signed adjugate of
  $H^{AB}=L(\zeta^{ab})$, as defined above. Then given any $x\in M$
  there is an open neighborhood $\tilde{U}$ of $x$ such that
  $\mathcal{H}_{AB}=I_AI_B$ for a nonvanishing cotractor $I_A$. We
  work locally on a fixed such set.

 Away from $\mathcal{D}(\ze^{ab})$ we can work in the scale $\nabla^g \in \bp$ where $\tau=\bdet(\zeta)$ and $g^{ab}=\operatorname{sgn}(\tau)\tau\ze^{ab}$. Then $\nabla^g$ preserves $g^{ab}$, $\tau$, and hence $\zeta$. Further, $\nabla^g$ preserves $\sigma$ since $\tau=\pm \sigma^2$. Thus it follows that, in the scale $\nabla^g$, $D_A\sigma=Y_A \sigma$ and by Proposition \ref{4.9} $H^{AB}=W_a^AW_b^B \operatorname{sgn}(\tau)\tau^{-1}g^{ab}$. So
\[
(D_A\sigma) H^{AB}= Y_A\sigma W_a^AW_b^B \operatorname{sgn}(\tau)\tau^{-1}g^{ab}=0.
\] Since $H^{AB}$ has a one-dimensional kernel and $X^AD_A\sigma=X^AI_A$ we see that $D_A\sigma = I_A$ on $\tilde{U} \cap (M \backslash \mathcal{D}(\ze^{ab}))$ so that $\mathcal{H}_{AB} = (D_{A}\sigma)(D_{B}\sigma)$ on $\tilde{U} \cap (M \backslash \mathcal{D}(\ze^{ab}))$. In order to conclude that $\mathcal{Z}(\sigma)=\mathcal{D}(\zeta^{ab})$ is a smoothly embedded hypersurface we will show that $\mathcal{H}_{AB} = (D_{A}\sigma)(D_{B}\sigma)$ also holds on $\mathcal{Z}(\sigma)\cap \tilde{U}$. Then $D_A\sigma$ nonvanishing implies that $\nabla_a\sigma \neq 0$ on $\mathcal{Z}(\sigma)$, and the result follows from the implicit function theorem.

The argument is now similar to that in the proof of Proposition \ref{4.8}.
The cotractor $I$, in any scale, is of the form
\begin{align*}
I_{A} = 
\left(
\begin{array}{c}
\sigma \\
\mu_{a} \\
\end{array} 
\right).
\end{align*}
We will show that $X^{B}\nabla_{c}I_{B} = 0$, which, according the formula for the tractor connection, implies that $\mu_a=\nabla_a\sigma$, whence  $I_{A} = D_{A}\sigma$. We begin by computing,

\begin{align*}
X^{A}X^{B}\nabla_{c}\mathcal{H}_{AB} & =  2X^{A}X^{B}I_{A}(\nabla_cI_{B}) =  2\sigma X^{B}\nabla_cI_{B}.  \\
\end{align*}
But we also have that
\begin{align*}
X^{A_n}X^{B_n}\nabla_{c}\mathcal{H}_{A_nB_n} & = \pm nX^{A_n}X^{B_n}\epsilon^2_{A_0\cdots A_nB_0\cdots B_n}H^{A_0B_0}\cdots H^{A_{n-2}B_{n-2}}(\nabla_{c}H^{A_{n-1}B_{n-1}})\\
& = \mp 2X^{A_n}X^{B_n}\epsilon^2_{A_0\cdots A_nB_0\cdots B_n}H^{A_0B_0}\cdots H^{A_{n-2}B_{n-2}} (X^{(A_{n-1}}W\indices{_{cE}^{B_{n-1})}_{F}}H^{EF}) \\
& = 0.
\end{align*}
Thus we see that $\sigma X^{B}\nabla_cI_{B}=0$.
Let $U \colonequals \{ x \in \tilde{U} : \sigma (x) \neq 0\}$, and suppose that $U$ is nonempty. Arguing exactly as in the proof of Proposition \ref{4.8} we conclude that 
\begin{align*}
I_{A} = 
\left(
\begin{array}{c}
\sigma \\
\nabla_{b}\sigma \\
\end{array} 
\right).
\end{align*}
We have already that $I_A$ is nonvanishing on $\tilde{U}$, and in
particular, on $\overline{U}$. Thus $\nabla_b\sigma \neq 0$ on
$\partial\overline{U}$. Thus $\mathcal{Z}(\sigma)$ is  a smooth
hypersurface. We can conclude that $I_A=D_A\sigma$ everywhere on the
local region where $\mathcal{H}_{AB}=I_AI_B$.

Then, since $\nabla_a \sigma \neq 0$ on $\mathcal{Z}(\sigma)$ it follows that $\mathcal{Z}(\sigma)$ is a smoothly embedded hypersurface where either $\sigma = \sqrt{\bdet(\zeta)}$ or $\sigma = -\sqrt{\bdet(\zeta)}$. If $\mathcal{Z}(\si)$ is orientable then a consistent choice of sign for $\sigma$ can be made in which case the local sections $I_A$ can be patched to form a global nonvanishing section and $\mathcal{Z}(\sigma)$ is a separating hypersurface.
So, provided that there exists a point $x \in M$ such that $\sigma(x) \neq 0$, the zero locus of $\sigma$ is nowhere dense so we conclude by the implicit function theorem that $\mathcal{Z}(\sigma)$ is a smoothly embedded hypersurface with $\sigma$ a defining density of weight 1 for the hypersurface. The claimed manifold decomposition follows trivially. If no such point exists, i.e if $U \colonequals \{ x \in M : \sigma (x) \neq 0\} = \varnothing$ then $\mathcal{Z}(\sigma)$ is all of $M$. 
\end{proof}

As a simple example to show that the hypersurface in the Theorem \ref{4.11} is, in general, only locally separating consider $H\subseteq \mathbb{R}^{n+1}$, an n-plane through the origin. Projectivizing gives $\mathbb{RP}^{n-1} \subseteq \mathbb{RP}^n$ as a non-separating hypersurface. 

In the following we carry forward the notation introduced above in the Theorem \ref{4.11} and its proof.
\begin{proposition}\label{4.12}
Let $(M,\bp)$ be a connected, projective 
manifold equipped with a solution $\ze$ of the metrizablility 
equation such that $L(\ze)$ has signature $(p,q,1)$ and $\varnothing \subsetneq \mathcal{D}(\ze)\subsetneq M$. Then $\mathcal{R}(\zeta)=n-1$ on an open dense set of $\Sigma \colonequals \mathcal{D}(\ze)$. Further, $\Sigma$ decomposes according to the strict sign of a canonical density $\hat{\tau} \in \Gamma(\mathcal{E}(2)|_{\Sigma})$
\[
\Sigma =  \coprod\limits_{i \in \{+,0,-\}} \Sigma_i
\]
with $\hat{\tau} >0$ on $\Sigma_+$ and $\hat{\tau} <0$ on $\Sigma_-$, and $\hat{\tau} = 0$ on $\Sigma_0$. Moreover $\Sigma_0$ is a smoothly embedded hypersurface in $\Sigma$.

\end{proposition}

\begin{proof}
  Claim 1: $\mathcal{R}(\ze^{ab})\geq n-2$ on $M$.

  \smallskip
  
\noindent Proof of Claim 1: This follows from the argument used  in the proof of part (iii) Theorem
\ref{4.6}.  

\medskip

\noindent   Claim 2: The set of points in $\Sigma$ for which $\mathcal{R}(\zeta^{ab})= n-2$ is a smoothly embedded hypersurface of $\Sigma$.

\smallskip

\noindent Proof of Claim 2: We work locally as in Theorem \ref{4.11}.
The section $I_A \in \Gamma(\mathcal{E}_A|_{\Sigma})$ gives a well
defined subbundle $I^{\perp} \subset \mathcal{E}^A|_{\Sigma}$ given by
all $V^A \in \mathcal{E}^A$ such that $V^AI_A=0$. Choose
$\overline{I}^A\in \Gamma(\mathcal{E}^A)$ such that
$I_A\overline{I}^A=1$. Then $\hat{\delta}^A_B \colonequals \delta^A_B
- I_B\overline{I}^A$ defines a projection from $\mathcal{E}^A$ to the
subbundle $I^{\perp}$.  Let a hat over a section denote its projection
via $\hat{\delta}$ to $I^{\perp}$ or its dual $I_{\perp}$. For example writing $H^{AB}:=L(\zeta)$, we have
$\hat{H}^{AB}\colonequals \hat{\delta}^A_C\hat{\delta}^B_DH^{CD}$,
but note that in this particular case $\hat{H}^{AB}=H^{AB}$ as
$I_CH^{CB}=0$. Observe that, along $\Sigma$, $X^A$ is in $I^{\perp}$
so also $\hat{X}^A=X^A$. Then we have a bundle $\mathcal{E}^{\Sigma}_A$,
where $\mathcal{E}^{\Sigma}_A := i^*(I_{\perp}) \cong
i^*(\overline{I}^{\perp})$ and $i^*$ denotes the pullback along the
inclusion $i:\Sigma \into M$ on $\Sigma$, and this evidently has the
following composition series,
\[
0 \rightarrow  \mathcal{E}^{\Sigma}_a(1) \xrightarrow{\hat{Z}^a_A}  \mathcal{E}^{\Sigma}_A \cong i^*(I_{\perp}) \xrightarrow{X^A}  \mathcal{E}^{\Sigma}(1) \rightarrow 0.
\]
where $\mathcal{E}^\Sigma_a$ is a notation for $T^*\Sigma$.  Observe
that $\hat{H}^{AB}|_{I_{\perp}}$ is nondegenerate hence invertible,
because $H^{AB}I_B=0$ and $H^{AB}$ has rank $n$. Denote the image of
the inclusion of its pointwise inverse into
$\mathcal{E}_{(AB)}|_\Sigma$ by $\hat{H}_{AB}$. Let $\hat{\tau}
\colonequals X^AX^B\hat{H}_{AB}$. 
Then $\mathcal{Z}(\hat{\tau})$ is the set of points in
$\Sigma$ for which $X$ is null with respect to $\hat{H}$.
We will show that there exists a non-zero  tangential vector field
$\xi^c \in \Gamma(T\Sigma)$ such that $\xi^c\nabla_c\hat{\tau}$ is
nonzero on $\mathcal{Z}(\hat{\tau})$ whence we will conclude that
$\mathcal{Z}(\hat{\tau})$ is a smoothly embedded hypersurface in
$\Sigma$. In the following we calculate along $\Sigma$. Now for any
$\xi^c \in \Gamma(T\Sigma)$ we have
\begin{align*}
\xi^c\nabla_c \hat{\tau} & = \xi^c\nabla_c(X^AX^B\hat{H}_{AB})\\
& = X^AX^B\xi^c\nabla_c\hat{H}_{AB} + 2X^A\hat{H}_{AB}\xi^c\nabla_cX^B\\
& = X^AX^B\xi^c\nabla_c\hat{H}_{AB} + 2X^A\hat{H}_{AB}\xi^c\hat{W}_c^B.\\
\end{align*}
Observe that, at any point $x \in \Sigma_0:=\mathcal{Z}(\hat\tau)$,
there exists a $\xi^c$ such that $2X^A\hat{H}_{AB}\xi^c\hat{W}_c^B\neq
0$ when $\hat{\tau}=0$. Otherwise $\hat{H}_{AB}$ would be degenerate,
a contradiction. Thus it suffices to show that
$X^AX^B\xi^c\nabla_c\hat{H}_{AB} =0$ on
$\mathcal{Z}(\hat{\tau})$. Noting that
$\hat{\delta}^C_A=\hat{H}_{AB}H^{BC}$ we see that
\begin{align*}
0 & = \xi^c\nabla_c{\delta}^C_A \\
& = \xi^c\nabla_c\hat{\delta}^C_A + \xi^c\nabla_c(I_A\overline{I}^C)  \\
& = \hat{H}_{AB}\xi^c\nabla_c H^{BC} + H^{BC}\xi^c\nabla_c\hat{H}_{AB} + \overline{I}^C\xi^c\nabla_cI_A + I_A\xi^c\nabla_c\overline{I}^C.
\end{align*}

Solving for
$\xi^c\nabla_c\hat{H}_{AD}$ and applying the ``Eastwood-Matveev formula'' \eqref{EMconn}
gives us
\begin{align*}
\xi^c\nabla_c\hat{H}_{AD} & = - \hat{H}_{CD}\hat{H}_{AB}\xi^c\nabla_cH^{BC} - \hat{H}_{CD}\overline{I}^C\xi^c\nabla_cI_A - \hat{H}_{CD}I_A\xi^c\nabla_c\overline{I}^C\\
& = \frac{2}{n}\hat{H}_{CD}\hat{H}_{AB}\xi^cX^{(B}W\indices{_{cE}^{C)}_F}H^{EF} -
\hat{H}_{CD}\overline{I}^C\xi^c\nabla_cI_A  - \hat{H}_{CD}I_A\xi^c\nabla_c\overline{I}^C. \\
\end{align*} 
Note that, by the proof of Theorem \ref{4.11}, $I_A=D_A\si$, hence on $\Sigma$ it follows immediately that $X^A\nabla_cI_A=0$ and $X^AI_A=0$. Thus contracting in $X^AX^D$ with $\xi^c\nabla_c\hat{H}_{AD}$ gives
\begin{align*}
X^AX^D\xi^c\nabla_c\hat{H}_{AD} & = \frac{2}{n}X^AX^D\hat{H}_{CD}\hat{H}_{AB}\xi^cX^{(B}W\indices{_{cE}^{C)}_F}H^{EF} \\
& = \frac{2}{n}\hat{\tau}\xi^cX^{D}\hat{H}_{CD}W\indices{_{cE}^{C}_F}H^{EF}.
\end{align*}
So $X^AX^D\xi^c\nabla_c\hat{H}_{AD} =0$ along $\mathcal{Z}(\hat{\tau})$. We conclude that $\nabla_c\hat{\tau} \neq 0$ on $\mathcal{Z}(\hat{\tau})$ implying that $\mathcal{Z}(\hat{\tau})$ is a smoothly embedded separating hypersurface of $\Sigma$ with the weight 2 density $\hat\tau$ as a defining density for it. The Proposition thus follows.
\end{proof}

Recall that $\si$ is given locally as the square root of $\bdet(\ze)$ or $-\bdet(\ze)$ and that $I_A=D_A\si$. In the following technical lemma we show that the free contravariant index in $Z_E^e\Omega\indices{_{ce}^{B}_{F}}H^{EF}$ is ``tangential" along the zero locus of $\sigma$ in the sense that it is in the kernel of $I_B$. 

We will shortly need the following technical result.
\begin{lemma} \label{4.13}
Let $(M,\bp)$ be a connected, projective 
manifold equipped with a solution $\ze$ of the metrizablility 
equation such that $H^{AB}\colonequals L(\ze)$ has signature $(p,q,1)$ and $\varnothing \subsetneq \mathcal{D}(\ze) \subsetneq M$. Then, locally, along the zero locus of $\sigma$, $I_BZ_E^e\Omega\indices{_{ce}^{B}_{F}}H^{EF}=0$.
\end{lemma}

\begin{proof}
Denote the signed adjugate of $H$ by $\mathcal{H}$. Given a point $x
\in M$ there exists an open set $U \subseteq M$ containing $x$ such
that $\mathcal{H}_{AB}=fI_AI_B$.  As in Proposition \ref{4.12} we absorb $f$
into $I$. We will first show that
$I_BZ_E^e\Omega\indices{_{ce}^{B}_{F}}H^{EF}=0$ on the set $U
\cap(\Sigma\backslash\Sigma_0)$. Along $\Sigma=\mathcal{D}(\ze)=\mathcal{Z}(\si)$ we
have that $\zeta^{ab}\nabla_{a}\sigma=0$. This follows at once from the fact
that $H^{AB}I_{B}=0$. But $\nabla_{a}\sigma$ is a (weighted) conormal to the
hypersurface, so $\zeta^{ab}$ is tangential along the hypersurface in
the sense that $\ze^{ab} \in \Gamma (S^2T\Sigma) \subset
\Gamma(S^2TM|_{\Sigma})$. Now we know the following:

\begin{align*}
X^{B_0} & \nabla_{c}\mathcal{H}_{A_0B_0}   = nX^{B_0}\epsilon^2_{A_0\cdots A_nB_0\cdots B_n}H^{A_1B_1}\cdots H^{A_{n-1}B_{n-1}}(\nabla_{c}H^{A_{n}B_{n}})  \\
& = -2X^{B_0}\epsilon^2_{A_0\cdots A_nB_0\cdots B_n}H^{A_1B_1}\cdots H^{A_{n-1}B_{n-1}}(X^{(A_{n}}W\indices{_{cE}^{B_{n})}_{F}}H^{EF})  \\
& = -X^{B_0}\epsilon^2_{A_0\cdots A_nB_0\cdots B_n}H^{A_1B_1}\cdots H^{A_{n-1}B_{n-1}}(X^{A_{n}}W\indices{_{cE}^{B_{n}}_{F}}H^{EF})  \\
& = (-1)^{n+1}\epsilon^2_{a_0\cdots a_{n-1}b_1\cdots b_{n}}Z^{a_0}_{A_0}\cdots Z^{a_{n-1}}_{A_{n-1}}Z^{b_1}_{B_1}\cdots Z^{b_n}_{B_{n}}H^{A_1B_1}\cdots H^{A_{n-1}B_{n-1}} W\indices{_{cE}^{B_{n}}_{F}}H^{EF}  \\
& = (-1)^{n+1}\epsilon^2_{a_0\cdots a_{n-1}b_1\cdots b_{n}}Z^{a_0}_{A_0}\cdots Z^{a_{n-1}}_{A_{n-1}}Z^{b_1}_{B_1}\cdots Z^{b_n}_{B_{n}}H^{A_1B_1}\cdots H^{A_{n-1}B_{n-1}} W^{B_n}_{h}W\indices{_{ce}^{h}_{f}}\ze^{ef}.
\end{align*}
Along $\Sigma$, this is equal to zero since, locally,

\begin{align*}
X^{B}\nabla_{c}\mathcal{H}_{AB}  & = X^{B}\nabla_{c}I_{A}I_{B} =  \sigma\nabla_{c}I_{A} + X^{B}I_{A}\nabla_{c}I_{B} = 0.
\end{align*} 
The last equality follows since $X^B\nabla_cI_B=0$, as established in the Proof of Theorem \ref{4.11}. Thus along $\Sigma$,  

\begin{align*}
0 & = \epsilon^2_{a_0\cdots a_nb_1\cdots b_n}\zeta^{a_1b_1}\cdots \zeta^{a_{n-1}b_{n-1}} W\indices{_{ce}^{b_{n}}_{f}}\zeta^{ef} \\
& = \pm \operatorname{adj}(\zeta)_{a_0b_n}W\indices{_{ce}^{b_{n}}_{f}}\zeta^{ef}, \\
\end{align*}
where we have used that $\epsilon^2_{a_0\cdots a_nb_1\cdots b_n}\zeta^{a_1b_1}\cdots \zeta^{a_{n-1}b_{n-1}}$ is the adjugate of $\zeta$, up to sign. But $\zeta$ has rank $n-1$ on the set $\Sigma \backslash \Sigma_0$, whence its adjugate is rank 1 and simple. So, locally,
 \begin{align*}
0 & =  \beta_{a}\beta_{b}W\indices{_{ce}^{b}_{f}}\zeta^{ef} 
\end{align*} 
for some nonvanishing 1-form $\beta$. It is clear that $\beta_{a}\zeta^{ab}=0$. Since $\zeta^{ab}$ has a rank $1$ kernel on $\Sigma \backslash \Sigma_0$ which contains both $\nabla_{a}\sigma$ and $\beta_{a}$, they must be proportional. Thus,
\begin{align*}
0 = (\nabla_{b}\sigma) W\indices{_{ce}^{b}_{f}} \zeta^{ef}.  
\end{align*} 
On $\Sigma$ this is equivalent to 
\begin{align*}
0 = I_{B}Z^e_E\Omega\indices{_{ce}^{B}_{F}} H^{EF}.  
\end{align*}
By continuity it follows that $I_BZ_E^e\Omega\indices{_{ce}^{B}_{F}}H^{EF}=0$ on the closure of this set, namely $\Sigma = \operatorname{cl}(\Sigma\backslash\Sigma_0)$. \\
\end{proof}

\begin{theorem} \label{4.14}
Let $(M,\bp)$ be a connected, projective 
manifold equipped with a solution $\ze$ of the metrizablility 
equation such that $L(\ze)$ has signature $(p,q,1)$ and $\varnothing \subsetneq \mathcal{D}(\ze) \subsetneq M$. Then  we have the partition of $M$ given by Theorem \ref{4.11} and  the following hold: 

$(i)$ $M_0 \colonequals \mathcal{D}(\ze)$ is totally geodesic and inherits a projective structure $\hat{\bp}$ whose projective densities agree with the restriction of the ambient projective densities to $M_0$.  

$(ii)$ $(M_\pm, g_\pm)$ are each scalar-flat, pseudo-Riemannian
manifolds with metric $g_{\pm}$ of the same signature as $\zeta$,
where $g_{\pm}^{ab} = \operatorname{sgn}(\tau)\tau
\zeta^{ab}|_{M_{\pm}}$ and $\tau\colonequals \text{\normalfont
  \bdet}(\ze^{ab})$. Moreover $\nabla^{g_\pm}\in\bp$, and if $M$ is
closed then $(M\backslash M_{\mp},\bp)$ is an order 1 projective
compactification of $(M_{\pm},\nabla^g)$ with boundary $M_0$.

$(iii)$ $(M_0,\hat{\bp})$ inherits a solution $\hat{\zeta}$ of the metrizability equation and $\Sigma \colonequals M_0$ decomposes according to the strict sign of the determinant, $\overline{\tau}$, of this solution such that $(\Sigma_{\pm},\hat{g}_{\pm})$ are pseudo-Riemannian for the metric $\hat{g}^{ab}_\pm \colonequals \operatorname{sgn}({\overline{\tau}})\overline{\tau}\hat{\ze}^{ab}$. Furthermore, $\nabla^{\hat{g_\pm}}\in \hat{\bp}$ and $(\Sigma \backslash \Sigma_{\mp}, \hat{\bp})$ is an order 2 projective compactification of $(\Sigma_{\pm},\nabla^{\hat{g}_\pm})$ with boundary $\Sigma_0$. Finally $\Sigma_0\colonequals \mathcal{D}(\hat{\ze})$, if nonempty, inherits a conformal structure of signature $(p-1,q-1)$.
 \end{theorem}

\begin{proof}
  $(i)$ We will show that, locally along the hypersurface $\Sigma \colonequals \mathcal{Z}(\sigma)$, $\xi^c\nabla^{\mathcal{T}}_cI_A \propto I_A$ for tangential vector fields $\xi^c \in \Gamma(T\Sigma)$. Then, looking at the slots of $\xi^c\nabla^{\mathcal{T}}_cI_A$ and $I_A$, it follows that $\nabla_a \sigma \propto \xi^c\nabla_c\nabla_a \sigma$. This shows $\Sigma$ is totally geodesic and so inherits a projective structure $(\Sigma, \hat{\bp}=\bp |_{{T\Sigma}})$. We work locally with notation as above.
  Since $H^{AB}\colonequals L(\ze^{ab})$ has rank $n$  and $H^{AB}I_B=0$, we need only show that $H^{AB}\xi^c\nabla^{\mathcal{T}}_cI_B=0$. We know that 
\[
0 = \nabla^{\mathcal{T}}_c(H^{AB}I_B) = \nabla^{\mathcal{T}}_c(H^{AB})I_B + H^{AB}\nabla^{\mathcal{T}}_cI_B,
\] 
so it is sufficient to show that $(\xi^c\nabla^{\mathcal{T}}_cH^{AB})I_B=0$ for all $\xi^c \in \Gamma(T\Sigma)$. By Lemma \ref{4.13} we know that $I_{B}Z^e_E\Omega\indices{_{ce}^{B}_{F}} H^{EF}$ vanishes on $\Sigma$. We also know that $X^BI_B=\sigma$ vanishes on $\Sigma$. Thus, along $\Sigma$, we have,
\begin{align*}
0 & = -\frac{2}{n}\xi^c X^{(A}\Omega\indices{_{cd}^{B)}_{E}} H^{DE}Z^d_DI_B = \xi^c(\nabla^{\mathcal{T}}_cH^{AB})I_B,
\end{align*}
as required.

Next we show that the projective densities on $\Sigma$, of a given weight, agree with the densities of the same weight on $M$ restricted to $\Sigma$ (cf. Proposition 12 of \cite{CG2}). The canonical conormal bundle $\mathcal{N}^* \subset T^*M |_{\Sigma}$ is defined as the annihilator of $T\Sigma$. Then for the density bundles we have that $\mathcal{E}(-n-1)|_{\Sigma} = \Lambda ^{n} T^*M |_{\Sigma} \cong \mathcal{N}^* \otimes \Lambda^{n-1} T^*\Sigma = \mathcal{N}^* \otimes \mathcal{E}^{\Sigma}(-n)$, where where bundles on $(\Sigma, \bp |_{T\Sigma})$ are denoted by a superscript (or subscript when convenient) $\Sigma$. 

Defining $n_a\colonequals \nabla_a\sigma$ gives a nowhere vanishing section of $\mathcal{N}^*(1) \cong \mathcal{E}(-n)|_{\Sigma} \otimes \mathcal{E}^{\Sigma}(n)$. This induces a canonical isomorphism $\mathcal{E}(-n)|_{\Sigma} \cong \mathcal{E}^{\Sigma}(-n)$, and hence $\mathcal{E}(-1)|_{\Sigma} \cong \mathcal{E}^{\Sigma}(-1)$. Although $I_A$ only exists locally for non-orientable $\Sigma$, densities are independent of the orientability of their base manifold and the above isomorphism of density bundles extends globally.

$(ii)$ Scalar flatness of the inherited metrics $g^{ab}_{\pm}=\operatorname{sgn}(\tau)\tau\ze^{ab}|_{\pm}$ on the open components follows from the fact that the scalar curvature is proportional to $\operatorname{det}({H^{AB}})$. That $g$ and $\ze$ have the same signature is clear. Since $\tau=\si^2$ it follows that $\nabla^g$ preserves $\sigma$. Of course $\nabla^{g_\pm}\in \bp$, so 
Then projective compactification of order 1 follows from the fact that $\sigma$ is a $1$-density that is a defining density for $M_0$ (as shown in Theorem \ref{4.11}). 

$(iii)$  Note that $H^{AB}I_B=0$ implies that $H^{AB}$ is nondegenerate on the quotient bundle $\mathcal{E}^{\Sigma}_A$ of $i^*\mathcal{E}_A$. Thus $L(\zeta)$ defines a nondegenerate bundle metric on $\mathcal{E}^{\Sigma}_A$. $H^{AB}I_B=0$ also implies that $\zeta^{ab}\nabla_a \sigma =0$ and $(\nabla_b \zeta^{ab})(\nabla_a \sigma)=0$ on $\Sigma$, since $\sigma=0$ along $\Sigma$. But since $\zeta$ has rank $1$ kernel on a dense open subset on $\Sigma$ and $\nabla_a \sigma$ is normal to $\Sigma$, it follows that $\zeta$ is nondegenerate on $T\Sigma$ on a dense open subset of $\Sigma$. 

Let $\hat{\nabla} \colonequals \nabla|_{T\Sigma}$ denote the
connection induced on $T\Sigma$ along the inclusion $i:\Sigma \into
M$, since $\Sigma$ is totally geodesic. Since
$(\nabla_b\si)\ze^{ab}=0$ and $\nabla\si$ is a conormal field, we let
$\hat{\ze}\in S^2T\Sigma$ denote the symmetric bilinear tensor along
$\Sigma$ given by the restriction of $\ze^{ab}$. We will show that $\hat{\ze}^{ab}$ is
a solution to the metrizability equation, that is
\[
\hat{\nabla}_c\hat{\ze}^{ab} = \frac{1}{n}\hat{\delta}_c^{a}\hat{\nabla}_d\hat{\ze}^{db} + \frac{1}{n}\hat{\delta}_c^{b}\hat{\nabla}_d\hat{\ze}^{ad},
\]
along $\Sigma$.
Let $n_a \colonequals \nabla_a \sigma$. We know that $n_a\ze^{ab}=0$ along $\Sigma$. We also know that, for tangential vector fields $\xi^c \in \Gamma(\mathcal{E}_{\Sigma}^c)$ that $\xi^c\nabla_cn_a \propto n_a$. So, along $\Sigma$, recalling the notation from (11),
\begin{align*}
0 & = \xi^c\nabla_c(n_a\ze^{ab}) = \ze^{ab}\xi^c\nabla_cn_a + n_a\xi^c\nabla_c\ze^{ab} = n_a\xi^c\nabla_c\ze^{ab} =-2n_a\xi^c\delta_c^{(a}\lambda^{b)}=-n_a\xi^{b}\lambda^{a}.
\end{align*}
So we see that $\xi^c\nabla_c\ze^{ab}$ and $\lambda^a$ are tangential along the hypersurface. 
Recall from $(9)$ that the metrizability equation is given by $\nabla_c\ze^{ab} = -2\delta_c^{(a}\lambda^{b)}$. Thus along the hypersurface we have $\xi^c\hat{\nabla}_c\hat{\ze}^{ab} = -2\xi^{(a}\hat{\lambda}^{b)}$, where $\hat{\lambda}^a = -\frac{1}{n}\hat{\nabla}_c\hat{\ze}^{ac}$. 

Then, along $\Sigma$, $\hat\ze$ induces a pseudo-Riemannian structure
$\hat{g}_\pm$, with Levi-Civita connection $\nabla^{g_\pm}\in \bp$, on
$\Sigma \backslash \mathcal{D}(\hat{\zeta})$ by $\hat{g}^{ab}=
\operatorname{sgn}(\overline{\tau})\overline{\tau}\hat{\ze}^{ab}$. By
a similar argument to Lemma \ref{4.4} we have that $\overline{\tau}
\colonequals \bdet(\hat{\ze}) $ and $\hat{\tau}=X^AX^BH_{AB}$ agree up
to multiplication by a nowhere zero function. Thus from Proposition
\ref{4.12} $\Sigma_0:=\mathcal{D}(\hat{\ze})$ is a smoothly
embedded hypersurface in $\Sigma$ with $\overline{\tau}$ a defining
density for it.
Moreover we have 
a decomposition $\Sigma = \Sigma_+
\cup \Sigma_0 \cup \Sigma_-$, according to the strict sign of
$\overline{\tau} \colonequals \bdet(\hat{\ze})$, or equivalently, the
strict sign of $\hat{\tau} \colonequals X^AX^B\hat{H}_{AB}$.
Since $\nabla^{\hat{g}_\pm}$ evidently preserves $\overline{\tau}$ on
$\Sigma \setminus \mathcal{D}(\hat{\zeta})$ it follows that $(\Sigma
\backslash \Sigma_{\mp}, \hat{\bp})$ is an order 2 projective
compactification of $(\Sigma_{\pm},\nabla^g)$. An argument analogous
to the proof of part $(iii)$ of Theorem \ref{4.6} shows that the weighted
conormal to $\Sigma_0$ is given by
\begin{align*}
\nabla_a \overline{\tau} = 2\operatorname{adj}(\hat{\ze})_{ac}\hat{\nabla}_e\hat{\ze}^{ce},
\end{align*}
and so $\nabla_a\overline{\tau}$ lies in, and hence spans, the kernel
of $\hat{\ze}^{ab}$, at each point of $\Sigma_0$. So
$\hat{\ze}^{ab}\in \Gamma(S^2T\Sigma_0(-2))$ is nondegenerate on
$T^*\Sigma_0$, whence it gives a conformal metric of the given signature claimed. (It is
straightforward to show that $\mathcal{E}(-2)|_{\Sigma_0}= \mathcal{E}_\Sigma[-2]$
where $(\mathcal{E}[-2])^{n-2}=((\Lambda^{n-2}T^*\Sigma_0)^2)$ is the
usual oriented line bundle said to have {\em conformal weight}
$-2(n-2)$, cf. \cite{CG1}.)
\end{proof}

Although not mentioned in the introduction, we note here that the
tractor bundles fit together nicely.
\begin{corollary}\label{4.15}
Let $(M,\bp)$ a projective manifold satisfying the conditions of Theorem \ref{4.14} with  $\mathcal{E}^{\Sigma}(1)\colonequals \mathcal{E}(1)|_{\Sigma}$. Then the tractor bundle   $\mathcal{E}^{\Sigma}_A = J^1\mathcal{E}^{\Sigma}(1)$ on the hypersurface $M_0\colonequals \Sigma$ fits into the following commutative diagram, along $\Sigma$, where $i:\Sigma \hookrightarrow M$ denotes the inclusion and $\hat{\pi}$ denotes the obvious subbundle projection.  \\

\begin{center}
\begin{tikzcd}%
	 & 0 \arrow[d] & 0 \arrow[d] \\
	 0 \arrow[r] & \mathcal{E}_{\Sigma} \arrow[r,"\sim" ] \arrow[d,hookrightarrow,"\nabla_a\sigma"] & \mathcal{E}_{\Sigma}  \arrow[d,hookrightarrow,"\nabla_a\sigma"]   \\
	0 \arrow[r] & i^*\mathcal{E}_a(1) \arrow[r,hookrightarrow,"Z^a_A"] \arrow[d,twoheadrightarrow,"\hat{\pi}"] & i^*\mathcal{E}_A = J^1\mathcal{E}(1)|_{\Sigma} \arrow[r,twoheadrightarrow,"X^A"]  \arrow[d,twoheadrightarrow,"\hat{\Pi}"] & i^*\mathcal{E}(1) \arrow[r] \isoarrow{d}  & 0 \\  
	0 \arrow[r] & \mathcal{E}^{\Sigma}_a(1) \arrow[r,hookrightarrow,"\tilde{Z}^a_A"] \arrow[d] & \mathcal{E}^{\Sigma}_A = J^1\mathcal{E}^{\Sigma}(1) \cong i^*(I^{\perp})^* \arrow[r,twoheadrightarrow,"X^A"] \arrow[d] & \mathcal{E}^{\Sigma}(1) \arrow[r] \arrow[d] & 0 \\
	& 0  & 0  & 0  \\
\end{tikzcd}%
\end{center}
\end{corollary}

\begin{proof}

Along $\Sigma$ the section $\nabla_a\sigma$ trivializes the weighted conormal bundle $\mathcal{N}^*(1)$ whence we identify it with the trivial bundle $\mathcal{E}_{\Sigma}$.
\end{proof} 


Theorem 5 of \cite{CG3} shows that if the interior of a manifold with
boundary is equipped with a pseudo-Riemannian metric whose Levi-Civita
connection does not extend to the boundary while its projective
structure does, and such that the generalized scalar curvature is
non-zero everywhere (on the manifold with boundary), then the metric is
projectively compact of order 2. Recall that the generalized scalar
curvature condition means that $\mathcal{R}(L(\ze^{ab}))=n+1$
everywhere.  We provide the analogue of that result for scalar flat
metrics.

\begin{corollary} \label{4.16}
Let $\overline{M}$ be an orientable, connected manifold with boundary $\partial M$ and interior $M$, equipped with a scalar-flat pseudo-Riemannian metric $g$ on $M$, such that its Levi-Civita connection $\nabla^{g}$ does not extend to any neighborhood of a boundary point, but the projective structure $\bp \colonequals [\nabla^g]$ does extend to the boundary. Let $\tau \colonequals \operatorname{vol}(g)^{-\frac{2}{n+2}}$. Then $\ze^{ab}\colonequals \tau^{-1}g^{ab}$ extends to the boundary. If $\mathcal{R}(L(\ze^{ab}))=n$ on $\overline{M}$, then $(M,g)$ is projectively compact of order $1$.
\end{corollary}

\begin{proof}
$\tilde{\nabla} L(\tau^{-1}g^{ab}) = 0$ on $M$ for $\tilde{\nabla}$
  the ``Eastwood-Matveev connection'' given by the left hand side of
  \eqref{EMconn}. But this connection is well-defined on all of
  $\overline{M}$ so we can extend $L(\tau^{-1}g^{ab})$ to a smooth
  parallel section of all of $\overline{M}$. As the projecting
  component, it follows that $\ze^{ab}\colonequals\tau^{-1}g^{ab}$
  smoothly extends to all of $\overline{M}$ as well. Clearly,
  $\mathcal{D}(\ze^{ab})=\partial M$, since otherwise $\nabla^{g}$
  would extend smoothly to the boundary, a contradiction. Now the
  result follows from Theorem \ref{4.14}.
\end{proof}

\end{document}